\documentclass[review]{siamart}

\usepackage[shortlabels]{enumitem}
\usepackage{bm}
\usepackage{amssymb}
\usepackage{mathtools}
\usepackage{pgf,tikz}
\usetikzlibrary{shapes.geometric}
\usetikzlibrary{arrows}
\usetikzlibrary{shapes.misc}
\tikzset{cross/.style={cross out, draw=black, minimum size=2*(#1-\pgflinewidth), inner sep=0pt, outer sep=0pt}, cross/.default={3pt}}
\usepackage{etoolbox}
\usetikzlibrary{shapes.geometric}
\usetikzlibrary{arrows}
\usepackage{amsfonts}
\usepackage{float}
\usepackage{graphicx}
\usepackage{epstopdf}
\usepackage{subcaption}
\usepackage{csquotes}
\ifpdf
\DeclareGraphicsExtensions{.eps,.pdf,.png,.jpg}
\else
\DeclareGraphicsExtensions{.eps}
\fi
\makeatletter
\patchcmd{\@addmarginpar}{\ifodd\c@page}{\ifodd\c@page\@tempcnta\m@ne}{}{}
\makeatother
\reversemarginpar

\usepackage{hyperref}

\graphicspath{{./FIGS/}}
\usepackage{color}
\definecolor{Pratik}{rgb}{1, 0, 0}
\definecolor{Tommaso}{rgb}{0, 0, 1}
\definecolor{Faycal}{rgb}{0, 1, 0}
\definecolor{Martin}{rgb}{0.7, 0.5, 0}

\newcommand{\vect}[1]{\mathbf{#1}}
\usepackage{algorithm}
\usepackage{algorithmic}
\newcommand{\fv}{\vect{f}}
\newcommand{\uv}{\vect{u}}
\newcommand{\vv}{\vect{v}}
\newcommand{\wv}{\vect{w}}
\newcommand{\rv}{\vect{r}}
\newcommand{\ev}{\vect{e}}
\newcommand{\av}{\mathbf{a}}
\newcommand{\yv}{\mathbf{y}}
\newcommand{\tv}{\mathbf{t}}
\newcommand{\sign}{{\rm sign}}

\newcommand{\RAS}{\text{RAS}}
\newcommand{\SRAS}{\text{SRAS}}
\newcommand{\JF}{J_\mathcal{F}}
\newcommand{\JFS}{J_{\overline{\mathcal{F}}}}

\newcommand{\FtwoL}{\mathcal{F}_{2L}}
\newcommand{\FtwoLS}{\overline{\mathcal{F}}_{2L}}
\newcommand{\TheTitle}{Linear and nonlinear substructured Restricted Additive Schwarz iterations and preconditioning} 
\newcommand{\TheAuthors}{F. Chaouqui, M. J. Gander, P. M. Kumbhar and T. Vanzan}

\title{{\TheTitle}}

\author{F. Chaouqui\thanks{Temple University
Philadelphia, USA ({\tt Faycal.Chaouqui@temple.edu})}, \and  M.J. Gander\thanks{Universit\'e de Gen\`eve, Switzerland ({\tt martin.gander@unige.ch})}, \and  P.M. Kumbhar\thanks{Karlsruher Institut f\"{u}r Technologie, Germany ({\tt pratik.kumbhar@kit.edu}),}
	\and T. Vanzan\thanks{CSQI Chair, \'{E}cole Polytecnique F\'{e}d\'{e}rale de Lausanne, Switzerland({\tt tommaso.vanzan@epfl.ch}).}}

\ifpdf
\hypersetup{
	pdftitle={\TheTitle},
	pdfauthor={\TheAuthors}
}
\fi

\begin{document}
	\maketitle

\begin{abstract}
Substructured domain decomposition (DD) methods have been
  extensively studied, and they are usually associated with
  nonoverlapping decompositions. We introduce here a
  substructured version of Restricted Additive Schwarz (RAS) which we
    call SRAS, and we discuss its advantages compared to the standard
  volume formulation of the Schwarz method when they are used both
    as iterative solvers and preconditioners for a Krylov
  method. To extend SRAS to nonlinear problems, we
  introduce SRASPEN (Substructured Restricted Additive Schwarz
  Preconditioned Exact Newton), where SRAS is used
  as a preconditioner for Newton's method. We study carefully the
  impact of substructuring on the convergence and performance of
  these methods as well as their implementations. We finally introduce
  two-level versions of nonlinear SRAS and SRASPEN. Numerical
  experiments confirm the advantages of formulating a Schwarz method
  at the substructured level.
  \end{abstract}
	


\section{Introduction}
  We consider a boundary value problem posed in a Lipschitz domain
  $\Omega\subset \mathbb{R}^d$, $d\in\left\{1,2,3\right\}$,
\begin{equation}\label{eq:BVP}
\begin{aligned}
\mathcal{L}(u)&=f,\quad \text{in } \Omega,\\
u&=0,\quad \text{on }\partial \Omega.
\end{aligned}
\end{equation}
 We assume that \eqref{eq:BVP} admits a unique solution in some Hilbert space 
 $\mathcal{V}$.
 If the boundary value problem is linear, a discretization of \eqref{eq:BVP}  
 with $N_v$ degrees of freedom leads to a linear system
\begin{equation}\label{eq:linear}
A\uv=\fv,
\end{equation}
where $A\in \mathbb{R}^{N_v\times N_v}$, $\uv\in V(\cong \mathbb{R}^{N_v})$, and $\fv\in V$.
If the boundary value problem is nonlinear, we obtain a nonlinear system
\begin{equation}\label{eq:nonlinear}
F(\uv)=0,
\end{equation}
where $F:V\rightarrow V$ is a nonlinear function and $\uv\in V$.
Several numerical methods have been proposed in the last decades for
the efficient solution of such boundary value problems, e.g.,
multigrid methods \cite{hackbusch2013multi,trottenberg2000multigrid}
and domain decomposition (DD) methods
\cite{ToselliWidlund,quarteroni1999domain}.  We will focus on DD
methods, which are usually divided into two distinct classes, that is
overlapping methods, which include the AS (Additive Schwarz) and RAS
(Restricted Additive Schwarz) methods
\cite{ToselliWidlund,cai1999restricted}, and nonoverlapping methods
such as FETI (Finite Element Tearing and Interconnect) and
Neumann-Neumann methods
\cite{farhat1991method,mandel1996balancing,klawonn2001feti}.
Concerning nonlinear problems, DD methods can be applied either as
nonlinear iterative methods, that is by just solving nonlinear
problems in each subdomain and then exchanging information between
subdomains as in the linear case
\cite{lui1999schwarz,lui2003monotone,cai1994domain}, or as
preconditioners to solve the Jacobian linear system inside Newton
iteration. In the latter case, the term Newton-Krylov-DD is employed,
where DD is replaced by the domain decomposition preconditioner used
\cite{10.1007/978-3-319-52389-7_19}.

An alternative is to use a DD method as a preconditioner for Newton's
method.  Preconditioning a nonlinear system $F(\uv)=0$ means that we aim
to replace the original nonlinear system with a new nonlinear system,
still having the same solution, but for which the nonlinearities are
more balanced and Newton's method converges faster
\cite{cai2002nonlinearly,Gandernonlinear}.  Seminal contributions in
nonlinear preconditioning have been made by Cai and Keyes in
\cite{cai2002nonlinearly,cai2001nonlinear}, where they introduced
ASPIN (Additive Schwarz Preconditioned Inexact Newton), which is
a left preconditioner. Extensions of this idea to Dirichlet-Neumann, FETI-DP and BDDC are
presented in \cite{chaouqui2021nonlinear,klawonn2014nonlinear}.  The development of good
preconditioners is not an easy task even in the linear case. One
useful strategy is to study efficient iterative methods, and then to
use the associated preconditioners in combination with Krylov methods
\cite{Gandernonlinear}. The same logical path paved the way to the
development of RASPEN (Restricted Additive Schwarz Preconditioned
  Exact Newton) in \cite{dolean2016nonlinear}, which in short
  applies Newton's method to the fixed point equation
defined by the nonlinear RAS iteration at convergence. All
these methods are left preconditioners. Right preconditioners are
usually based on the concept of nonlinear elimination, presented in
\cite{lanzkron1996analysis}, and they are very efficient as shown in
\cite{gong2019nonlinear,gong2017nonlinear,cai2011inexact,doi:10.1137/19M1307184}.  While left preconditioners aim to
transform the original nonlinear function into a better behaved one,
right preconditioners aim to provide a better initial guess for the
next outer Newton iteration.

Nonoverlapping methods are sometimes called substructuring methods (a
term borrowed from Przemieniecki's work
\cite{przemieniecki1963matrix}), as in these methods the unknowns in
the interior of the nonoverlapping domains are eliminated through
static condensation so that one needs to solve a smaller system
involving only the degrees of freedom on the interfaces between the
nonoverlapping subdomains, like in the more recent hybridizable
  discontinuous Galerkin finite element methods, where each element
  forms a subdomain \cite{cockburn2004characterization}.  However, it
is also possible to write an overlapping method, such as Lions' Parallel
Schwarz Method (PSM), which is equivalent to RAS
  \cite{gander2008schwarz}, in substructured form, even though this
approach is much less common in the literature. For a two
subdomain decomposition, a substructuring procedure applied to the PSM
is carried out in \cite[Section 5]{gander2006optimized},
\cite[Section 3.4]{gander2012methodes} and \cite{ciaramella-vanzan2}, while for a many subdomain
decomposition with cross-points we refer to
\cite{ciaramella-vanzan,ciaramella-vanzan-spectral}. In this
particular framework, the substructured unknowns are now the degrees
of freedom located on the portions of a subdomain boundary that lie in
the interior of another subdomain; that is where the overlapping DD
method takes the information to compute the new
iterate. We emphasize that, at a given iteration $n$,
any iterative DD method (overlapping or nonoverlapping) needs only a
few values of $\uv^{n}$ to compute the new approximation $\uv^{n+1}$. The
major part of $\uv^{n}$ is useless.

In this manuscript, we define a substructured version of RAS,
that is we define an iterative scheme based on RAS which acts
only over unknowns that are located on the portions of a subdomain
boundary that lie in the interior of another subdomain. We study
in detail the effects that such a substructuring procedure has on
RAS when the latter is applied either as an iterative solver or
as a preconditioner to solve linear and nonlinear boundary value
problems. Does the substructured iterative version converge faster
than the volume one? Is the convergence of GMRES affected by
substructuring? What about nonlinear problems when instead of
preconditioned GMRES we rely on preconditioned Newton?  We prove that substructuring does not influence the convergence of the
iterative methods both in the linear and nonlinear case, by showing that at each iteration, the restriction on the interfaces
of the volume iterates coincides with the iterates of the
substructured iterative method. This equivalence of iterates does not
hold anymore when considering preconditioned GMRES. Specifically, our
study shows that GMRES should always be applied to the substructured
system, since it is computationally less expensive, requiring to perform orthogonalization on a much smaller space,
and thus needs also less memory. In contrast to the linear case,
we prove that the nonlinear preconditioners RASPEN and
SRASPEN (Substructured RASPEN) for Newton produce the same
iterates once these are restricted to the interfaces. However, SRASPEN has again more favorable properties when assembling and
solving the Jacobian matrices at each Newton iteration. Finally,
we also extend the work in
\cite{ciaramella-vanzan,ciaramella-vanzan2,ciaramella-vanzan-spectral}
defining substructured two-level methods to the nonlinear case, where both smoother and coarse correction are defined directly on the
interfaces between subdomains.

This paper is organized as follows: we introduce in Section \ref{sec:domain} the mathematical setting with the domain, subdomains and operators defined on them.  In Section
\ref{Sec:linearcase}, devoted to the linear case, we study the
effects of substructuring on RAS and on GMRES applied to the
preconditioned system. In Section \ref{Sec:nonlinear}, we
extend our analysis to nonlinear boundary value problems. Section
\ref{Sec:two-level} contains two-level substructured methods for the
nonlinear problems. Finally Section \ref{sec:Num_section} provides
extensive numerical tests to corroborate the framework proposed.

\section{Notation}\label{sec:domain}

Let us decompose the domain $\Omega$ into $N$ nonoverlapping subdomains $\Omega_j$, that is $\Omega = \bigcup_{j \in \mathcal{J}} \Omega_j$ with $\mathcal{J}:=\{1,2,\dots,N\}$.
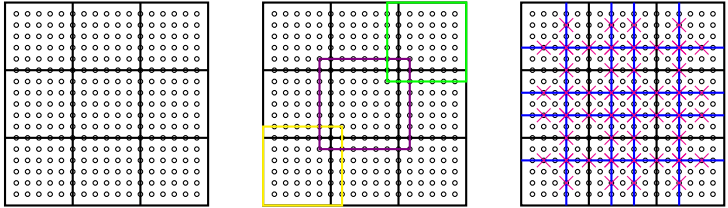
\begin{figure}
\centering
\begin{tikzpicture}[scale=0.15]
\draw [thick,black] (1,1) rectangle (19,19);
\draw [thick,black] (1,7) rectangle (19,7);
\draw [thick,black] (1,13) rectangle (19,13);
\draw [thick,black] (7,1) rectangle (7,19);
\draw [thick,black] (13,1) rectangle (13,19);
\foreach \x in {2,3,4,5,6,7,8,9,10,11,12,13,14,15,16,17,18}
    \foreach \y in {2,3,4,5,6,7,8,9,10,11,12,13,14,15,16,17,18}
      {
        \draw (\x,\y) circle (0.2cm) ;
      }
\end{tikzpicture}\qquad
\begin{tikzpicture}[scale=0.15]
\draw [thick,black] (1,1) rectangle (19,19);
\draw [thick,black] (1,7) rectangle (19,7);
\draw [thick,black] (1,13) rectangle (19,13);
\draw [thick,black] (7,1) rectangle (7,19);
\draw [thick,black] (13,1) rectangle (13,19);
\foreach \x in {2,3,4,5,6,7,8,9,10,11,12,13,14,15,16,17,18}
    \foreach \y in {2,3,4,5,6,7,8,9,10,11,12,13,14,15,16,17,18}
      {
        \draw (\x,\y) circle (0.2cm) ;
      }
\draw [thick,yellow] (1,1) rectangle (8,8);  
\draw [thick,violet] (6,6) rectangle (14,14); 
\draw [thick,green] (12,12) rectangle (19,19);      
\end{tikzpicture}\qquad 
\begin{tikzpicture}[scale=0.15]
\draw [thick,black] (1,1) rectangle (19,19);
\draw [thick,black] (1,7) rectangle (19,7);
\draw [thick,black] (1,13) rectangle (19,13);
\draw [thick,black] (7,1) rectangle (7,19);
\draw [thick,black] (13,1) rectangle (13,19);
\foreach \x in {2,3,4,5,6,7,8,9,10,11,12,13,14,15,16,17,18}
    \foreach \y in {2,3,4,5,6,7,8,9,10,11,12,13,14,15,16,17,18}
      {
        \draw (\x,\y) circle (0.2cm) ;
      }
\draw [thick,blue] (1,5) -- (19,5);     
\draw [thick,blue] (1,9) -- (19,9);
\draw [thick,blue] (1,11) -- (19,11);  
\draw [thick,blue] (1,15) -- (19,15);   
\draw [thick,blue] (5,1) -- (5,19);     
\draw [thick,blue] (9,1) -- (9,19);
\draw [thick,blue] (11,1) -- (11,19); 
\draw [thick,blue] (15,1) -- (15,19); 

 \foreach \x in {3,5,7,9,11,13,15,17}
    \foreach \y in {5,9,11,15}
      {
     \draw (\x,\y) node[cross,magenta] {};
      }
       \foreach \y in {3,5,7,9,11,13,15,17}
    \foreach \x in {5,9,11,15}
      {
     \draw (\x,\y) node[cross,magenta] {};
      }

\end{tikzpicture}
\caption{The domain $\Omega$ is divided into nine nonoverlapping subdomains (left). The center panel shows how the diagonal nonoverlapping subdomains are enlarged to form overlapping subdomains. On the right, we denote the unknowns represented in $\overline{V}$ (blue line) and the unknowns of a coarse space of $\overline{V}$ (red crosses).}
\label{fig:subdomain}
\end{figure}
The nonoverlapping subdomains $\Omega_j$ are then enlarged to obtain
subdomains $\Omega^\prime_j$ which form an overlapping decomposition
of $\Omega$.  For each subdomain $\Omega^\prime_j$, we define $V_j$ as
the restriction of $V$ to $\Omega^\prime_j$, that is $V_j$ collects
the degrees of freedom on $\Omega^\prime_j$. Further, we introduce the
classical restriction and prolongation operators $R_j:V\rightarrow
V_j$, $P_j:V_j\rightarrow V$, and the restricted prolongation
operators $\widetilde{P}_j:V_j\rightarrow V$. We assume that these
operators satisfy
\begin{equation}\label{eq:assumption}
R_jP_j=I_{V_j},\quad \mbox{and} \quad\sum_{j\in \mathcal{J}} \widetilde{P}_jR_j = I,
\end{equation}
where $I_{V_j}$ is the identity on $V_j$ and $I$ is the identity on $V$. 

We now define the substructured skeleton. In the following, we use the
notation introduced in \cite{CHS2}.  For any $j \in \mathcal{J}$, we
define the set of neighboring indices $N_j :=\{ \ell \in \mathcal{J}
\, : \, \Omega^\prime_j \cap \partial \Omega^\prime_\ell \neq
\emptyset \}$.  Given a $j \in \mathcal{J}$, we introduce the
substructure of $\Omega^\prime_j$ defined as $S_j := \bigcup_{\ell \in
  N_j} \bigl(\partial \Omega^\prime_\ell \cap \Omega^\prime_j\bigr)$,
that is the union of all the portions of $\partial \Omega^\prime_\ell$
with $\ell \in N_j$.  The substructure of the whole domain $\Omega$ is
defined as $S:=\bigcup_{j \in \mathcal{I}}\overline{S_j}$. A graphical
representation of $S$ is given in Figure \ref{fig:subdomain} for a
decomposition of a square into nine subdomains.  We now introduce the
space $\overline{V}$ as the trace space of $V$ onto the substructure
$S$, i.e. $\overline{V}$ collects the $\overline{N}$ degrees of
freedom which lie on $S$.  Associated to $\overline{V}$, we consider
the restriction operator $\overline{R}:V\rightarrow \overline{V}$ and
a prolongation operator $\overline{P}:\overline{V}\rightarrow V$.  The
restriction operator $\overline{R}$ takes an element $v\in V$ and
restricts it to the skeleton $S$. The prolongation operator
$\overline{P}$ extends an element $v\in \overline{V}$ to the global
space $V$. In our numerical experiments, $\overline{P}$ extends
an element $v_S\in \overline{V}$ by zero in $\Omega\setminus S$.
However, we can consider several different prolongation operators.
How this extension is done is not crucial as we will use
$\overline{P}$ inside a domain decomposition algorithm, and thus only
the values on the skeleton $S$ will play a role. Hence, as of now, we
will need only one assumption on the restriction and prolongation
operator, namely
\begin{equation}\label{eq:assumption1}
\overline{R}\overline{P}=\overline{I},
\end{equation}
where $\overline{I}$ is the identity over $\overline{V}$.

\section{The linear case}\label{Sec:linearcase}
In this section, we focus on the linear problem $A\vect{u}=\vect{f}$. After defining a substructured variant of RAS called SRAS, we prove the equivalence between RAS and SRAS. Then, we study in detail how GMRES performs if applied to the volume preconditioned system or the substructured system.
\subsection{Linear iterative methods}

To introduce our analysis, we recall the classical definition of
RAS to solve the linear system \eqref{eq:linear}. RAS starts
from an approximation $\vect{u}^0$ and computes for $n=1,2,\dots$,
\begin{equation}\label{eq:RAS_linear}
\vect{u}^n=\vect{u}^{n-1}+\sum_{j\in \mathcal{J}} \widetilde{P}_j A^{-1}_j R_j\left(\vect{f}-A\vect{u}^{n-1}\right),
\end{equation} 
where $A_j:=R_jAP_j$, that is, we use exact local solvers.
Let us now rewrite the iteration \eqref{eq:RAS_linear} in an equivalent form using the hypothesis in \eqref{eq:assumption} and the definition of $A_j$,
\begin{eqnarray}
\vect{u}^n &=& \sum_{j\in \mathcal{J}} \widetilde{P}_j R_j \vect{u}^{n-1}+\sum_{j\in \mathcal{J}} \widetilde{P}_j A^{-1}_j R_j\left(\vect{f}-A\vect{u}^{n-1}\right) \nonumber\\
&=&\sum_{j\in \mathcal{J}} \widetilde{P}_j A^{-1}_j \left( A_j R_j\vect{u}^{n-1} +R_j\left(\vect{f}-A\vect{u}^{n-1}\right)\right) \nonumber \\
&=&\sum_{j\in \mathcal{J}} \widetilde{P}_jA^{-1}_jR_j \left(\vect{f}-A\left(I-P_jR_j\right)\vect{u}^{n-1}\right) \nonumber\\
&=&:G^{\RAS}(\vect{u}^{n-1}) \label{eq:RAS2}.
\end{eqnarray}
We emphasize that $\left(P_jR_j-I\right)\vect{u}^{n-1}$ contains non-zero
elements only outside subdomain $\Omega^\prime_j$, and in particular
the terms $A\left(P_jR_j-I\right)\vect{u}^{n-1}$ represent precisely the
boundary conditions for $\Omega^\prime_j$ given the old approximation
$\vect{u}^{n-1}$. This observation suggests that RAS, like most domain
decomposition methods, can be written in substructured
form. Indeed, despite iteration \eqref{eq:RAS2} being written in
volume form, involving the entire vector $\vect{u}^{n-1}$, only
very few elements of $\vect{u}^{n-1}$ are needed to compute the new
approximation $\vect{u}^{n}$. For further details about a substructured
formulation of the parallel Schwarz method at the continuous level, we
refer to \cite{gander2012methodes} for the two subdomain case, and
\cite{ciaramella-vanzan,ciaramella-vanzan2,ciaramella-vanzan-spectral}
for a general decomposition into several subdomains with cross points.

In Section \ref{sec:domain}, we introduced the substructured space
$\overline{V}$ geometrically, but we can also provide an algebraic
characterization using the RAS operators $R_j$ and $P_j$.  We
consider \[\mathcal{K}:=\left\{ k\in \left\{1,\dots,N_v\right\}:
\exists j \in \left\{1,\dots,N\right\} \text{ such that }
R_jA(\vect{e}_k-P_jR_j\vect{e}_k)\neq 0\right\},\] that is, $\mathcal{K}$ is the set
of indices such that the canonical vectors $\vect{e}_k$ represent a Dirichlet
boundary condition at least for a subdomain, and its complement
$\mathcal{K}^c:=\left\{1,\dots,N_v\right\}\setminus \mathcal{K}$. The
cardinality of $\mathcal{K}$ is $|\mathcal{K}|=:\overline{N}$. We can
thus introduce \[\widehat{V}:=\left\{\vect{v}\in \mathbb{R}^{N_v}: \text{ if
}j\notin \mathcal{K} \text{ then }
v_j=0\right\}=\text{span}\left\{\vect{e}_k\right\}_{k\in \mathcal{K}}\subset
\mathbb{R}^{N_v}.\] Finally $\overline{R}$ is the Boolean restriction
operator, mapping a vector of $\mathbb{R}^{N_v}$ onto a vector of
$\mathbb{R}^{\overline{N}}$, keeping only the indices in
$\mathcal{K}$. Hence, $\overline{V}:=\text{Im}\overline{R}(\cong
\mathbb{R}^{\overline{N}})$ and $\overline{P}=\overline{R}^\top$.

To define SRAS, we need one more assumption on the restriction
and prolongation operators, namely
\begin{equation}\label{eq:assumption2}
\overline{R}M^{-1}A=\overline{R}M^{-1}A \overline{PR},
\end{equation}
where $M^{-1}$ is the preconditioner for RAS, formally defined as
\begin{equation}\label{def:M}
M^{-1}:=\sum_{j\in \mathcal{J}} \widetilde{P}_jA^{-1}_jR_j.
\end{equation}
Heuristically, this assumption means that the operator $\overline{P}\overline{R}$ preserves all the information needed by $G^\RAS$ (defined in \eqref{eq:RAS2}) to compute correctly the values of the new iterate on the skeleton $S$. Indeed a direct calculation shows that \eqref{eq:assumption2} is equivalent to the condition
\[\overline{R}G^\RAS(\vect{u})=\overline{R}G^\RAS(\overline{P}\overline{R}\vect{u}).\]

Given a substructured approximation $\vect{v}^0\in \overline{V}$, for $n=1,2,\dots$, we define SRAS  as
\begin{equation}\label{eq:SRAS}
\vect{v}^n=G^\SRAS(\vect{v}^{n-1}),\quad \text{where}\quad  G^\SRAS(\vect{v}):=\overline{R}G^\RAS(\overline{P}\vect{v}).
\end{equation}
RAS and SRAS are tightly linked, but when are they
equivalent? Clearly, we must impose some conditions on $\overline{P}$
and $\overline{R}$. The next theorem shows that 
  assumption \eqref{eq:assumption2} is in fact sufficient for equivalence.

\begin{theorem}[Equivalence between RAS and SRAS]\label{Th:Equivalence}
Assume that the operators $\overline{R}$ and $\overline{P}$ satisfy \eqref{eq:assumption2}. Given an initial guess $\vect{u}^0\in V$ and its substructured restriction $\vect{v}^0:=\overline{R} \vect{u}^0\in \overline{V}$, define the sequences $\left\{ \vect{u}^n\right\}$ and $\left\{\vect{v}^n\right\}$ such that 
\[\vect{u}^n=G^\RAS(\vect{u}^{n-1}),\quad \vect{v}^n=G^\SRAS(\vect{v}^{n-1}).\]
Then, $\overline{R}\vect{u}^{n}=\vect{v}^n$ for every iteration $n\geq 1$.
\end{theorem}
\begin{proof}
We prove this statement for $n=1$ by a direct calculation. Taking
the restriction of $\vect{u}^1$ we have
\[\overline{R}\vect{u}^1=\overline{R}G^\RAS(\vect{u}^0)=\overline{R}G^\RAS(\overline{P}\overline{R}\vect{u}^0)=\overline{R}G^\RAS(\overline{P}\vect{v}^0)=G^\SRAS(\vect{v}^0)=\vect{v}^1,\]
where we used assumption \eqref{eq:assumption2}, and the definition of
$\vect{v}^0$ and $G^{\text{SRAS}}$.  For a general $n$, the proof is obtained
by induction.
\end{proof}

\subsection{Linear preconditioners for GMRES}\label{Sec:GMRES}

It is well known that any stationary iterative method should be
used in practice as a preconditioner for a Krylov method,
  since the Krylov method finds in general a much better residual
  polynomial with certain optimality properties, compared to the
  residual polynomial of the stationary iteration, see
  e.g. \cite[Section 4.1]{CiaramellaGander2021book}.  The preconditioner
associated to RAS is $M^{-1}$ and is defined in
\eqref{def:M}. The preconditioned volume system then reads
\begin{equation}\label{eq:precond_vol}
M^{-1}A\uv=M^{-1}\fv.\\
\end{equation}
To discover the preconditioner associated with SRAS, we consider
the fixed point limit of \eqref{eq:SRAS},
\begin{equation}\label{eq:fixedpoint}
\begin{aligned}
\vv&=G^\SRAS(\vv)=\overline{R}G^\RAS(\overline{P}\vv)=\overline{R}\left(\overline{P}\vv+\sum_{j\in \mathcal{J}}\widetilde{P}_jA_j^{-1}R_j(\fv-A\overline{P}\vv)\right)\\
&=\vv +\overline{R}\sum_{j\in \mathcal{J}}\widetilde{P}_jA_j^{-1}R_j\fv - \overline{R}\sum_{j\in \mathcal{J}}\widetilde{P}_jA_j^{-1}R_j A\overline{P}\vv\\
&=\vv + \overline{R}M^{-1}\fv-\overline{R}M^{-1}A\overline{P}\vv,
\end{aligned}
\end{equation}
where in the second line we used the identity $\overline{R}\overline{P}=I_S$.
We can thus consider the preconditioned substructured system
\begin{equation}\label{eq:precond_sub}
\overline{R}M^{-1}A\overline{P}\vv=\overline{R}M^{-1}\fv.
\end{equation}
It is then natural to ask how a Krylov method like GMRES performs if
applied to \eqref{eq:precond_vol}, compared to \eqref{eq:precond_sub}.

Let us consider an initial guess in volume $\uv^0$, its restriction $\vv^0:=\overline{R}\uv^0$ and the initial residuals $\rv^0:=M^{-1}(\fv-A\uv^0)$, $\overline{\rv}^0:=\overline{R}M^{-1}(\fv-A\overline{P}\vv^0)$.
Then GMRES applied to the preconditioned systems \eqref{eq:precond_vol} and \eqref{eq:precond_sub} looks for solutions in the affine Krylov spaces
\begin{equation}\label{eq:Krylov}
\begin{aligned}
\uv^0+\mathcal{K}_k(M^{-1}A,\rv^0)&:=\uv^0+\text{ span}\left\{\rv^0,M^{-1}A\rv^0,\dots,(M^{-1}A)^{k-1}\rv^0\right\} \\
\vv^0+\mathcal{K}_k(\overline{R}M^{-1}A\overline{P},\overline{\rv}^0)&:=\vv^0+\text{ span}\left\{\overline{\rv}^0,\overline{R}M^{-1}A\overline{P}\overline{\rv}^0,\dots,(\overline{R}M^{-1}A\overline{P})^{k-1}\overline{\rv}^0\right\},
\end{aligned}
\end{equation}
where $k\geq 1$.  The two Krylov spaces are tightly linked, as Theorem
\ref{thm:Krylov} below will show. To prove it, we need the
following Lemma.

\begin{lemma}\label{Lemma:powerk}
If the restriction and prolongation operators $\overline{R}$ and $\overline{P}$ satisfy \eqref{eq:assumption2} then for $k\geq 1$,
\begin{equation}\label{eq:assumption2_generalk}
\overline{R}\left(M^{-1}A\right)^{k}=\left(\overline{R}M^{-1}A\overline{P}\right)^k\overline{R}.
\end{equation}
\end{lemma}
\begin{proof}
Multiplying the second equation in \eqref{eq:assumption2} from the
right by $M^{-1}A$ we get
\begin{eqnarray*}
\overline{R}M^{-1}AM^{-1}A&=&\overline{R}M^{-1}A\overline{P}\overline{R}M^{-1}A=\overline{R}M^{-1}A\overline{P}\overline{R}M^{-1}A\overline{P}\overline{R},
\end{eqnarray*}
where in the second equality we have used once more \eqref{eq:assumption2}. Using induction, one gets for every $k\geq 1$,
\[
\overline{R}\left(M^{-1}A\right)^{k}=\left(\overline{R}M^{-1}A\overline{P}\right)^k\overline{R},
\]
and this completes the proof.
\end{proof}

\begin{theorem}[Relation between RAS and SRAS Krylov subspaces]\label{thm:Krylov}
Let us consider operators $\overline{R}$ and $\overline{P}$ satisfying \eqref{eq:assumption2}, an initial guess $\uv^0\in V$, its restriction $\vv^0:=\overline{R}\uv^0\in \overline{V}$ and the residuals $\rv^0:=M^{-1}(\fv-A\uv^0)$, $\overline{\rv}^0:=\overline{R}M^{-1}(\fv-A\overline{P}\vv^0)$. Then for every $k\geq 1$, we have
\begin{equation}\label{eq:ResKrylov}
\vv^0+\mathcal{K}_k(\overline{R}M^{-1}A\overline{P},\overline{\rv}^0)=\overline{R}(\uv^0+\mathcal{K}_k(M^{-1}A,\rv^0)).
\end{equation}
\end{theorem}
\begin{proof}
First, due to \eqref{eq:assumption2} we have \[\overline{R}\rv^0=\overline{R}M^{-1}(\fv-A\uv^0)=\overline{R}M^{-1}(\fv-A\overline{P}\overline{R}\uv^0)=\overline{R}M^{-1}(\fv-A\overline{P}\vv^0)=\overline{\rv}^0.\]
Let us now show the first inclusion. If $\vv\in  \overline{R}\left( \uv^0 + \mathcal{K}_k(M^{-1}A,\rv^0)\right)$, then 
$\vv= \overline{R} \uv^0+ \overline{R}\sum_{j=0}^{k-1} \gamma_j \left(M^{-1}A\right)^j\rv^0,$
for some coefficients $\gamma_j$.
Using Lemma \ref{Lemma:powerk}, we can rewrite $\vv$ as
\begin{eqnarray*}
\vv &=& \vv^0+ \sum_{j=0}^{k-1} \gamma_j \overline{R}\left(M^{-1}A\right)^j \rv^0 =\vv^0+ \sum_{j=0}^{k-1} \gamma_j \left(\overline{R}M^{-1}A\overline{P}\right)^j\overline{R} \rv^0 \\
&=& \vv^0+ \sum_{j=0}^{k-1} \gamma_j \left(\overline{R}M^{-1}A\overline{P}\right)^j\overline{\rv}^0
\in  \vv^0+\mathcal{K}_k(\overline{R}M^{-1}A\overline{P},\overline{\rv}^0),
\end{eqnarray*}
 and thus $ \overline{R}\left( \uv^0 + \mathcal{K}_k(M^{-1}A,\rv^0)\right) \subset \vv^0+\mathcal{K}_k(\overline{R}M^{-1}A\overline{P},\overline{\rv}^0)$.
Similarly if $\wv\in \vv^0+\mathcal{K}_k(\overline{R}M^{-1}A\overline{P},\overline{\rv}^0)$ then
\begin{eqnarray*}
\wv &=&\vv^0+\sum_{j=0}^{k-1} \gamma_j \left(\overline{R}M^{-1}A\overline{P}\right)^k\overline{\rv}^0=\overline{R} \uv^0+\sum_{j=0}^{k-1} \gamma_j \left(\overline{R}M^{-1}A\overline{P}\right)^j\overline{R}\rv^0\\
&=&\overline{R} \uv^0+\sum_{j=0}^{k-1} \gamma_j \overline{R}\left(M^{-1}A\right)^{j}\rv^0,
\end{eqnarray*}
thus $\wv\in \overline{R}\left( \uv^0 + \mathcal{K}_k(M^{-1}A,\rv^0)\right)$ and we achieve the desired relation \eqref{eq:ResKrylov}.
\end{proof} 

Theorem \ref{thm:Krylov} shows that the restriction to the
substructure of the affine volume Krylov space of RAS coincides
with the affine substructured Krylov space of SRAS. One could
then wonder if the restrictions of the iterates of GMRES applied to
the preconditioned volume system \eqref{eq:precond_vol} coincide with
the iterates of GMRES applied to the preconditioned substructured
system \eqref{eq:precond_sub}.  However, this does not turn out
to be true. Nevertheless, we can further link the action of GMRES on
these two preconditioned systems.

It is well known, (see e.g \cite[Section 6.5.1]{saad2003iterative}), that GMRES applied to \eqref{eq:precond_vol} and \eqref{eq:precond_sub} generates a sequence of iterates $\left\{ \uv^k\right\}_k$ and $\left\{\vv^k\right\}_k$ such that
\begin{equation}\label{eq:lsq_vol}
\uv^k=\text{argmin}_{\widetilde{\uv}^k\in \uv^0+\mathcal{K}_k(M^{-1}A,\rv^0)} \|M^{-1}\fv-M^{-1}A\widetilde{\uv}^k\|_2,
\end{equation}
and
\begin{equation}\label{eq:lsq_sub}
\vv^k=\text{argmin}_{\widetilde{\vv}^k\in \vv^0+\mathcal{K}_k(\overline{R}M^{-1}A\overline{P},\overline{\rv}^0)} \|\overline{R}M^{-1}\fv-\overline{R}M^{-1}A\overline{P}\widetilde{\vv}^k\|_2.
\end{equation}
The iterates $\uv^k$ and $\vv^k$ can be characterized using orthogonal and
Hessenberg matrices obtained with the Arnoldi iteration. In
particular, the k-th iteration of Arnoldi provides orthogonal matrices
$Q_k,Q_{k+1}$ and a Hessenberg matrix $H_k$ such that
$M^{-1}AQ_k=Q_{k+1}H_k$, and the columns of $Q_j$ form an orthonormal
basis for the Krylov subspace $\mathcal{K}_k(M^{-1}A,\rv^0)$.  Using
these matrices, one writes $\uv^k$ as $\uv^k=\uv^0+Q_k\av$, where $\av
\in \mathbb{R}^k$ is the solution of the least squares problem
\label{eq:lsq_vol2}
\begin{equation}
\av=\text{argmin}_{\widetilde{\av}\in \mathbb{R}^k}\|Q_k(\|\rv^0\|_2 \ev_1 -H_k\widetilde{\av})\|_2 =\text{argmin}_{\widetilde{\av}\in \mathbb{R}^k}\|\|\rv^0\|_2 \ev_1 -H_k\widetilde{\av}\|_2,
\end{equation}
and $\ev_1$ is the canonical vector of $\mathbb{R}^{k+1}$.
Similarly, one characterizes the vector $\vv^k$ as $\vv^k=\vv^0+\overline{Q}_k\yv$ such that 
\begin{equation}\label{eq:lsq_sub2}
\yv=\text{argmin}_{\widetilde{\yv}\in \mathbb{R}^k}\|\|\overline{\rv}^0\|_2 \ev_1 -\overline{H}_k\widetilde{\yv}\|_2,
\end{equation}
where $\overline{Q}_k, \overline{H}_k$ are the orthogonal and Hessenberg matrices obtained through the Arnoldi method applied to the matrix $\overline{R}M^{-1}A\overline{P}$.

The next Theorem provides a link between the volume least square problem \eqref{eq:lsq_vol} and the substructured one \eqref{eq:lsq_sub}.
\begin{theorem}
Under the hypothesis of Theorem \ref{thm:Krylov}, the k-th iterate of GMRES applied to \eqref{eq:precond_sub} is equal to $\vv^k=\vv^0+\overline{Q}_k\yv=\vv^0+\overline{R}Q_k\tv$, where $\yv$ satisfies \eqref{eq:lsq_sub2} while
\begin{equation}\label{eq:mu}
\tv:=\text{argmin}_{\widetilde{\tv}\in \mathbb{R}^k}\|\overline{R}Q_{k+1}(\|\rv^0\|_2\ev_1-H_k\widetilde{\tv})\|_2.
\end{equation}
\end{theorem}
\begin{proof}
It is clear that $\vv^k=\vv^0+\overline{Q}_k\yv=\vv^0+\overline{R}Q_k\tv$ as the first equality follows
from standard GMRES literature (see e.g \cite[Section 6.5.1]{saad2003iterative}). The second equality follows from  Theorem \eqref{thm:Krylov} as we have shown that $\vv^0+\mathcal{K}_k(\overline{R}M^{-1}A\overline{P},\overline{\rv}^0)=\overline{R}(\uv^0+\mathcal{K}_k(M^{-1}A,\rv^0))$. Thus the columns of $\overline{R}Q_k$ form an orthonormal basis of $\mathcal{K}_k(\overline{R}M^{-1}A\overline{P},\overline{\rv}^0)$ and hence, $\vv^k$ can be expressed as a linear combination of the columns of $\overline{R}Q_k$ with coefficients in the vector $\tv\in \mathbb{R}^k$ plus $\vv^0$. We are then left to show \eqref{eq:mu}. We have
\begin{eqnarray*}
\min_{\widetilde{\vv}^k \in \vv^0+\mathcal{K}_k(\overline{R}M^{-1}A\overline{P},\overline{\rv}^0)}& & \|\ \overline{R}M^{-1}\fv-\overline{R}M^{-1}A\overline{P} \widetilde{\vv}^k\|_2\\
&=&\min_{\widetilde{\tv}\in\mathbb{R}^k} \|\ \overline{R}M^{-1}\fv-\overline{R}M^{-1}A\overline{P}(\vv^0+\overline{Q}_k\widetilde{\tv})\|_2\\
&=&\min_{\widetilde{\yv}\in\mathbb{R}^k} \|\ \overline{R}M^{-1}\fv-\overline{R}M^{-1}A\overline{P}\overline{R}\uv^0 -\overline{R}M^{-1}A\overline{P}\overline{Q}_k\widetilde{\yv})\|_2.
\end{eqnarray*}
Using the relation Im$(\overline{R}Q_k)=$Im$(\overline{Q}_k)$, Lemma \eqref{Lemma:powerk}, the Arnoldi relation $M^{-1}AQ_k=Q_{k+1}H_k$ and that $\rv^0$ coincides with the first column of $Q_k$ except for a normalization constant, we conclude
\begin{equation}
\begin{aligned}
&\min_{\widetilde{\tv}\in\mathbb{R}^k} \|\ \overline{R}M^{-1}\fv-\overline{R}M^{-1}A\overline{P}\overline{R}\uv^0 -\overline{R}M^{-1}A\overline{P}\overline{R}Q_k\widetilde{\tv})\|_2\\
&=\min_{\widetilde{\tv}\in\mathbb{R}^k} \|\ \overline{R}\rv^0 -\overline{R}M^{-1}AQ_k\widetilde{\tv})\|_2\\
&=\min_{\widetilde{\tv}\in\mathbb{R}^k} \|\ \overline{R}Q_{k+1}(\|\rv^0\|_2 \ev_1 -H_k\widetilde{\tv})\|_2,
\end{aligned}
\end{equation}
and this completes the proof.
\end{proof}

Few comments are in order here. First, GMRES applied to \eqref{eq:precond_sub} converges in maximum $\overline{N}$ iterations as the preconditioned matrix $\overline{R}M^{-1}A\overline{P}$ has size $\overline{N}\times\overline{N}$.
Second, Theorem \eqref{thm:Krylov} states that $\overline{R}(\uv^0+\mathcal{K}_{\overline{N}}(M^{-1}A,\rv^0))$  already contains the exact substructured solution, that is the exact substructured solution lies in the restriction of the volume Krylov space after $\overline{N}$ iterations. Theoretically, if one could get the exact substructured solution from $\overline{R}(\uv_0+\mathcal{K}_{\overline{N}}(M^{-1}A,\rv^0))$, then $\overline{N}$ iterations of GMRES applied to \eqref{eq:precond_vol}, plus an harmonic extension of the substructured data into the subdomains, would be sufficient to get the exact volume solution. 

On the other hand, we can say a bit more analyzing the structure of $M^{-1}A$. Using the splitting $A=M-N$, we have $M^{-1}A=I-M^{-1}N$. A direct calculation states $\mathcal{K}_k(M^{-1}A,\rv^0)=\mathcal{K}_k(M^{-1}N,\rv^0)$ by 
using the relation
  \[(M^{-1}A)^k=(I-M^{-1}N)^k=\sum_{j=0}^k {k\choose j} (-1)^j (M^{-1}N)^j\quad \forall k\geq 1,\]
that is the Krylov space generated by $M^{-1}A$ is equal to the Krylov space generated by the RAS iteration matrix for error equation. We denote this linear operator with $G^\RAS_0$ which is defined as in \eqref{eq:RAS2} with $\fv=0$. We now consider the orthogonal complement $\widehat{V}^\perp:=(\text{span}\left\{\ev_k\right\}_{k\in \mathcal{K}})^\perp=\text{span}\left\{\ev_i\right\}_{i\in \mathcal{K}^c}$, and $\text{dim}(\widehat{V}^\perp)=N_v-\overline{N}$.
Since for every $\vv\in \widehat{V}^\perp$, it holds that $R_jA(I-P_jR_j)\vv=0$, we can conclude that $\widehat{V}^\perp \subset \text{ker}\left(G_0^\RAS\right)$.
  Using the rank-nullity theorem, we obtain 
\[\text{dim\big(Im} (G_0^\RAS)\big) + \text{dim}\big(\text{Ker}\left(G_0^\RAS\right)\big)=N_v\implies \text{dim\big(Im}\left( G_0^\RAS\right)\big) \leq \overline{N},\] 
hence GMRES applied to the preconditioned volume system encounters a lucky Arnoldi breakdown after at most $\overline{N}+1$ iterations (in exact arithmetic). This rank argument can be used for the substructured preconditioned system as well. Indeed as
$\overline{R}M^{-1}A\overline{P}=\overline{I}-\overline{R}M^{-1}N\overline{P}$,
the substructured Krylov space is generated by the matrix
$\overline{R}M^{-1}N\overline{P}$, whose rank is equal to the rank of
$M^{-1}N$, that is the rank of $G_0^\RAS$.

Heuristically, choosing a zero initial guess, $\rv^0:=M^{-1}\fv$  corresponds to a solution of subdomains problem with the correct right hand side, but with zero Dirichlet boundary conditions along the interfaces of each subdomain. Thus, GMRES applied to \eqref{eq:precond_vol} needs only to find the correct boundary conditions for each subdomain, and this can be achieved in at most $\overline{N}$ iterations as Theorem \ref{thm:Krylov} shows.

Finally, each GMRES iteration on \eqref{eq:lsq_sub} is computationally less expensive than a GMRES iteration on \eqref{eq:lsq_vol} as the orthogonalization of the Arnoldi method is carried out in a much smaller space. From the memory point of view, this implies that 
GMRES needs to store shorter vectors. Thus, a saturation of the memory is less likely, and restarted versions of GMRES may be avoided.

\section{The nonlinear case}\label{Sec:nonlinear}
In this section, we study iterative and preconditioned domain
decomposition methods to solve the nonlinear system
\eqref{eq:nonlinear}.

\subsection{Nonlinear iterative methods}

RAS can be generalized to solve the nonlinear equation
\eqref{eq:nonlinear}. To show this, we introduce the solution
operators $G_j$ which are defined through
\begin{equation}\label{eq:definition_Gi}
R_jF(P_j G_j(\uv)+ (I-P_jR_j)\uv)=0,
\end{equation}
where the operators $R_j$ and $P_j$ are defined in Section \ref{sec:domain}.
Nonlinear RAS for $N$ subdomains then reads
\begin{equation}\label{eq:RASnonlinear}
\uv^n=\sum_{j\in \mathcal{J}} \widetilde{P}_j G_j(\uv^{n-1}).
\end{equation} 
It is possible to show that \eqref{eq:RASnonlinear} reduces to \eqref{eq:RAS2} if $F(\uv)$ is a linear function: assuming that $F(\uv)=A\uv-\fv$, equation \eqref{eq:definition_Gi} becomes
\begin{eqnarray*}
R_j F\left(P_j G_j\left(\uv^{n-1}\right)+ \left(I-P_jR_j\right)\uv^{n-1}\right)&=& R_j\left(A\left(P_jG_j\left(\uv^{n-1}\right)+\left(I-P_jR_j\right)\uv^{n-1}\right)-\fv\right)\\
&=&A_jG_j\left(\uv^{n-1}\right)+R_j\left(A\left(I-P_jR_j\right)\uv^{n-1}-\fv\right)=0,
\end{eqnarray*}
which implies $G_j\left(\uv^{n-1}\right)=A_j^{-1}R_j\left(\fv-A\left(I-P_jR_j\right)\uv^{n-1}\right)$, and thus \eqref{eq:RASnonlinear} reduces to \eqref{eq:RAS2}.

Similarly to the linear case, we define nonlinear SRAS: with
\begin{equation}\label{eq:definition_GSi}
\overline{G}_j(\vv^{n-1}):=\overline{R}\widetilde{P}_jG_j\left(\overline{P}\vv^{n-1}\right),
\end{equation}
 we obtain the nonlinear substructured iteration
\begin{equation}\label{eq:SRASnonlinear}
\vv^n=\overline{R}\sum_{j\in \mathcal{J}} \widetilde{P}_j G_j(\overline{P}\vv^{n-1})=\sum_{j\in \mathcal{J}}\overline{G}_j(\vv^{n-1}),
\end{equation}
which is the nonlinear counterpart of \eqref{eq:SRAS}.
 
The same calculations of Theorem \ref{Th:Equivalence} allow one to obtain an equivalence result between nonlinear RAS and nonlinear SRAS.
\begin{theorem}[Equivalence between nonlinear RAS and SRAS]\label{Th:Equivalence_nonlinear}
Assume that the operators $\overline{R}$ and $\overline{P}$ satisfy $\overline{R}\sum_{j\in \mathcal{J}}\widetilde{P}_j G_j(\uv)=\overline{R}\sum_{j\in \mathcal{J}}\widetilde{P}_j G_j(\overline{P}\overline{R}\uv)$. Let us consider an initial guess $\uv^0\in V$ and its substructured restriction $\vv^0:=\overline{R} \uv^0\in \overline{V}$, and define the sequences $\left\{ \uv^n\right\}$, $\left\{\vv^n\right\}$ such that 
\[\uv^n=\sum_{j\in \mathcal{J}} \widetilde{P}_j G_j(\uv^{n-1}),\quad \vv^n=\sum_{j\in \mathcal{J}}\overline{G}_j(\vv^{n-1}).\]
Then for every $n\geq 1$, $\overline{R}\uv^{n}=\vv^n$.
\end{theorem}

\subsection{Nonlinear preconditioners for Newton's method}
In the manuscript \cite{dolean2016nonlinear}, it has been proposed to
use the fixed point equation of nonlinear RAS as a
preconditioner for Newton's method, in a spirit that goes back to
\cite{cai2001nonlinear,cai2002nonlinearly}. This method has been
called RASPEN (Restricted Additive Schwarz Preconditioned Exact
Newton) and it consists in applying Newton's method to the
fixed point equation of nonlinear RAS, that is,
\begin{equation}\label{eq:RASPEN}
\mathcal{F}(\uv)=\uv-\sum_{j\in \mathcal{J}} \widetilde{P}_j G_j(\uv)=0.
\end{equation} 
For a comprehensive discussion of this method, we refer to
\cite{dolean2016nonlinear}.  As done in \eqref{eq:fixedpoint} for the
linear case, we now introduce a substructured variant of RASPEN
and we call it SRASPEN (Substructured Restricted Additive
Schwarz Preconditioned Exact Newton). SRASPEN is obtained by
applying Newton's method to the fixed point equation of nonlinear
SRAS, that is,
\begin{equation*}
\overline{\mathcal{F}}(\vv):=\vv-\sum_{j\in \mathcal{J}} \overline{G}_j(\vv)=0.
\end{equation*} 
One can verify that the above equation $\overline{\mathcal{F}}(\vv)=0$ can also be written as
\begin{equation}\label{eq:SRASPEN}
\overline{\mathcal{F}}(\vv)= \overline{R}\overline{P}\vv-\sum_{j\in \mathcal{J}} \overline{R}\widetilde{P}_j G_j(\overline{P}\vv)=\overline{R}\mathcal{F}(\overline{P}\vv)=0.
\end{equation}
This formulation of SRASPEN provides its relation with RASPEN and
simplifies the task of computing the Jacobian of SRASPEN.

\subsubsection{Computation of the Jacobian and implementation details}\label{Sec:Impl}
To apply Newton's method, we need to compute the Jacobian of SRASPEN. Let $J_{\mathcal{F}}(\wv)$ and $J_{\overline{\mathcal{F}}}(\wv)$ denote the action of the Jacobian of RASPEN and SRASPEN on a vector $\wv$. Since these methods are closely related, indeed $\overline{\mathcal{F}}(\vv)=\overline{R}\mathcal{F}(\overline{P}\vv)$, we can immediately compute the Jacobian of $\overline{\mathcal{F}}$ once we have the Jacobian of $\mathcal{F}$, using the chain rule, $J_{\overline{\mathcal{F}}}(\vv)=\overline{R}J_{\mathcal{F}}(\overline{P}\vv)\overline{P}$. The Jacobian of $\mathcal{F}$ has been derived in \cite{dolean2016nonlinear} and we report here the main steps for the sake of completeness.
Differentiating equation \eqref{eq:RASPEN} with respect to $\uv$ leads to
\begin{equation}\label{eq:Jac_SRASPEN}
J_\mathcal{F}(\uv):=\frac{d\mathcal{F}}{d\uv}(\uv)=I-\sum_{j\in \mathcal{J}}\widetilde{P}_j\frac{d G_j}{d\uv}(\uv).
\end{equation}
Recall that the local inverse operators $G_j:V\rightarrow V_j$ are defined in 
equation \eqref{eq:definition_Gi} as the 
solutions of $R_jF(P_j G_j(\uv)+ (I-P_jR_j)\uv)=0$. Differentiating this relation yields
\begin{equation}\label{eq:diff_RASPEN}
\frac{dG_j}{d\uv}(\uv)=R_j-\left(R_j J\left(\uv^{(j)}\right)P_j\right)^{-1}R_j J\left(\uv^{(j)}\right), 
\end{equation}
where $\uv^{(j)}:=P_jG_j(\uv)+(I-P_jR_j)\uv$ is the volume solution vector in subdomain $j$ and $J$ is the Jacobian of the original nonlinear function $F$. Combining the above equations \eqref{eq:Jac_SRASPEN}-\eqref{eq:diff_RASPEN} and defining $\widetilde{\uv}^{(j)}:=P_jG_j(\overline{P}\vv)+(I-P_jR_j)\overline{P}\vv$, we get
\begin{equation}\label{eq:Jacobian_RASPEN}
J_{\mathcal{F}}(\uv)= \left(\sum_{j\in \mathcal{J}} \widetilde{P}_j\left(R_j J\left(\uv^{(j)}\right)P_j\right)^{-1}R_j J\left(\uv^{(j)}\right)\right),
\end{equation}
and
\begin{equation}\label{eq:Jacobian_SRASPEN}
J_{\overline{\mathcal{F}}}(\vv)= \overline{R}\left(\sum_{j\in \mathcal{J}} \widetilde{P}_j\left(R_j J\left(\widetilde{\uv}^{(j)}\right)P_j\right)^{-1}R_j J(\widetilde{\uv}^{(j)})\right)\overline{P},
\end{equation}
where we used the assumptions $\sum_{j\in \mathcal{J}}
\widetilde{P}_jR_j=I$ and $\overline{R}\overline{P}=I_S$.  We remark
that to assemble $J_{\mathcal{F}}(\uv)$ or to compute its action on a
given vector, one needs to calculate $J\left(\uv^{(j)}\right)$, that is,
evaluate the Jacobian of the original nonlinear function $F$ on the
subdomain solutions $\uv^{(j)}$. The subdomain solutions $\uv^{(j)}$ are
obtained evaluating $\mathcal{F}(\uv)$, that is performing one step of
RAS with initial guess equal to $\uv$.  A smart implementation can
use the local Jacobian matrices $R_j J\left(\uv^{(j)}\right)P_j$ that
are already computed by the inner Newton solvers while solving the
nonlinear problem on each subdomain, and hence no extra cost is
required to assemble this term. Further, the matrices $R_j
J\left(\uv^{(j)}\right)$ are different from the local Jacobian matrices
at very few columns corresponding to the degrees of freedom on the
interfaces and thus it suffices to only modify those specific entries. In a
non-optimized implementation, one can also directly evaluate the Jacobian of $F$
on the subdomain solutions $\uv^{(j)}$, without relying on already
computed quantities. Concerning $J_{\overline{\mathcal{F}}}(\vv)$, we
emphasize that $\widetilde{\uv}^{(j)}$ is the volume subdomain solution
obtained by substructured RAS starting from a substructured function
$\vv$.  Thus, like $\uv^{(j)}$, $\widetilde{\uv}^{(j)}$ is readily available
in Newton's iteration after evaluating the function
$\overline{\mathcal{F}}$.

From the computational point of view, \eqref{eq:Jacobian_RASPEN} and \eqref{eq:Jacobian_SRASPEN} have several implications as the substructured Jacobian $J_{\overline{\mathcal{F}}}$ is a matrix of dimension $\overline{N}\times \overline{N}$ where $\overline{N}$ is the number of unknowns on $S$, and thus is a much smaller matrix than $J_{\mathcal{F}}$, whose size is $N_v\times N_v$, with  
$N_v$ the number of unknowns in volume. 
On the one hand, if one prefers to assemble the Jacobian matrix, either because one wants to use a direct solver or because one wants to recycle the Jacobian for few iterations, then SRASPEN dramatically reduces the cost of the assembly of the Jacobian matrix.
On the other hand, if one prefers to use a Krylov method such as GMRES, then according to the discussion in Section \ref{Sec:GMRES},
SRASPEN better exploits the properties of the underlying domain decomposition method, and saves computational time by permitting to perform the orthogonalization in a much smaller space.
Further implementation details and a more extensive comparison are available in the numerical section \ref{sec:Num_section}.

\subsubsection{Convergence analysis of RASPEN and SRASPEN}\label{Sec:Comparison}
Theorem \ref{Th:Equivalence_nonlinear} gives an equivalence between nonlinear RAS and nonlinear SRAS. Are RASPEN and SRASPEN equivalent? Does Newton's method behave differently if applied to the volume or to the substructured fixed point equation, like it happens with GMRES (see Section \ref{Sec:GMRES})? 
In this section, we aim to answer these questions by discussing the convergence properties of the exact Newton's method applied to $\mathcal{F}$ and $\overline{\mathcal{F}}$.

Let us recall that, given two approximations $\uv^0$ and $\vv^0$, the exact Newton's method computes for $n\geq 1$,
\begin{equation*}
\uv^n=\uv^{n-1}-\left(\JF\left(\uv^{n-1}\right)\right)^{-1}\mathcal{F}\left(\uv^{n-1}\right)\quad \mbox{and}\quad \vv^n=\vv^{n-1}-\left(\JFS\left(\vv^{n-1}\right)\right)^{-1}\overline{\mathcal{F}}(\vv^{n-1}),
\end{equation*}
where $\JF\left(\uv^{n-1}\right)$ and $\JFS\left(\vv^{n-1}\right)$ are the Jacobian matrices respectively of $\mathcal{F}$ and $\overline{\mathcal{F}}$ evaluated at $\uv^{n-1}$ and $\vv^{n-1}$.
In this paragraph, we do not need a precise expression for $\JF$ and $\JFS$. However we recall that, the definition $\overline{\mathcal{F}}(\vv)=\overline{R}\mathcal{F}(\overline{P}\vv)$ and the chain rule derivation provides us the relation $\JFS(\vv)=\overline{R}\JF(\overline{P}\vv)\overline{P}$. If the operators $\overline{R}$ and $\overline{P}$ were square matrices, we would immediately obtain that RASPEN and SRASPEN are equivalent, due to the affine invariance theory for Newton's method \cite{deuflhard2010newton}. However, in our case, $\overline{R}$ and $\overline{P}$ are rectangular matrices and the mapping between spaces is of different dimensions. Nevertheless, in the following theorem, we show that RASPEN and SRASPEN provide the same iterates restricted to the interfaces under further assumptions on $\overline{R}$ and $\overline{P}$, which is a direct generalization of \eqref{eq:assumption2} to the nonlinear case.

\begin{theorem}[Equivalence between RASPEN and SRASPEN]\label{Th:Equivalence_RASPEN_SRASPEN}
Assume that the operators $\overline{R}$ and $\overline{P}$ satisfy 
\begin{equation}\label{eq:Assumption_RASPEN}
\overline{R}\mathcal{F}(\uv)=\overline{R}\mathcal{F}(\overline{P}\overline{R}\uv)=\overline{\mathcal{F}}(\overline{R}\uv).
\end{equation}
Given an initial guess $\uv^0\in V$ and its substructured restriction $\vv^0:=\overline{R} 
\uv^0\in \overline{V}$, define the sequences $\left\{ \uv^n\right\}$ and $\left\{\vv^n\right\}$ such that 
\[\uv^n=\uv^{n-1}-\left(\JF\left(\uv^{n-1}\right)\right)^{-1}\mathcal{F}\left(\uv^{n-1}\right)\quad \mbox{and} \quad \vv^n=\vv^{n-1}-\left(\JFS\left(\vv^{n-1}\right)\right)^{-1}\overline{\mathcal{F}}\left(\vv^{n-1}\right).\]
Then for every $n\geq 1$, $\overline{R}\uv^{n}=\vv^n$.
\end{theorem}
\begin{proof}
We first prove the equality $\overline{R}\uv^{1}=\vv^1$ by direct
calculations.  Taking the restriction of the RASPEN iteration, we
obtain
\begin{equation}\label{eq:proof1}
 \overline{R}\uv^{1}= \overline{R}\uv^0-\overline{R}\left(\JF\left(\uv^{0}\right)\right)^{-1}\mathcal{F}\left(\uv^{0}\right)=\vv^0-\overline{R}\left(\JF\left(\uv^{0}\right)\right)^{-1}\mathcal{F}\left(\uv^{0}\right).
\end{equation}
Now, due to the definition of $\overline{\mathcal{F}}$ and of $\vv^0$, and to the assumption \eqref{eq:Assumption_RASPEN}, we have 
\begin{equation}\label{eq:proof2}
\overline{\mathcal{F}}\left(\vv^{0}\right)=\overline{R}\mathcal{F}\left(\overline{P}\vv^0\right)=\overline{R}\mathcal{F}\left(\overline{P}\overline{R}\uv^0\right)=\overline{R}\mathcal{F}\left(\uv^0\right).
\end{equation}
Further, taking the Jacobian of assumption \eqref{eq:Assumption_RASPEN}, we have $\overline{R} J_{\mathcal{F}}(\uv^0)=J_{\overline{\mathcal{F}}} (\overline{R}\uv^0)\overline{R}$, which simplifies by taking the inverse of the Jacobians to
\begin{equation}\label{eq:proof3}
  \overline{R}\left( J_{\mathcal{F}}(\uv^0)\right)^{-1}=\left( J_{\overline{\mathcal{F}}}(\overline{R}\uv^0)\right)^{-1}\overline{R}.
\end{equation}
Finally substituting relations \eqref{eq:proof2} and \eqref{eq:proof3} into \eqref{eq:proof1} leads to
\begin{eqnarray*}
\overline{R}\uv^{1} &=&\vv^0-\overline{R}\left(\JF\left(\uv^{0}\right)\right)^{-1}\mathcal{F}\left(\uv^{0}\right) =\vv^0-\left( J_{\overline{\mathcal{F}}}(\overline{R}\uv^0)\right)^{-1}\overline{R}\mathcal{F}\left(\uv^{0}\right)\\
&=&\vv^0-\left( J_{\overline{\mathcal{F}}}(\vv^0)\right)^{-1}\overline{\mathcal{F}}\left(\vv^{0}\right) =\vv^1,
\end{eqnarray*}
 and the general case is obtained by induction.
\end{proof}

\section{Two-level nonlinear methods}\label{Sec:two-level}
RAS and SRAS can be generalized to two-level iterative schemes.
This has already been treated in detail for the linear case in
\cite{ciaramella-vanzan,ciaramella-vanzan-spectral,ciaramella-vanzan2}. In
this section, we introduce two-level variants for nonlinear RAS
  and SRAS, and also for the associated RASPEN and SRASPEN.
\subsection{Two-Level iterative methods}
To define a two-level method, we introduce a coarse space $V_0\subset V$, a restriction operator $R_0:V\rightarrow V_0$ and an interpolation operator $P_0:V_0\rightarrow V$. The nonlinear system $F$ can be projected onto the coarse space $V_0$, defining the coarse nonlinear function $F_0\left(\uv_0\right):=R_0F\left(P_0\uv_0\right)$, for every $\uv_0\in V_0$. Due to this definition, it follows immediately that $J_{F_0}\left(\uv_0\right)=R_0 J_F\left(P_0\uv_0\right)P_0$, $\forall \uv_0 \in V_0$. To compute a coarse correction we rely on the FAS approach \cite{FAS}. Given a current approximation $\uv$, the coarse correction $C_0(\uv)$ is computed as the solution of
\begin{equation}\label{eq:coarseequation}
F_0(C_0(\uv)+R_0\uv)=F_0(R_0\uv)-R_0F(\uv).
\end{equation} 
Two-level nonlinear RAS is described by Algorithm \ref{alg:Two-level_RAS} and it consists of a coarse correction followed by one iteration of nonlinear RAS.
\begin{algorithm}[t]
\setlength{\columnwidth}{\linewidth}
\caption{Two-level nonlinear RAS}
\begin{algorithmic}[1]
\STATE  Solve the coarse problem $F_0\left(\yv\right)=F_0\left(R_0\uv^k\right)-R_0F\left(\uv^k\right)$ and set $C_0\left(\uv^k\right)=\yv-R_0\uv^k$.
\STATE Add the coarse correction to the current iterate, $\uv^{k+\frac{1}{2}}=\uv^k+P_0C_0\left(\uv^k\right)$.
\STATE Compute one step of nonlinear RAS, $\uv^{k+1}=\sum_{j\in \mathcal{J}}\widetilde{P}_jG_j\left(\uv^{k+\frac{1}{2}}\right)$.
\STATE Repeat steps 1 to 3 until convergence.
\end{algorithmic}\label{alg:Two-level_RAS}
\end{algorithm}

We now focus on its substructured counterpart. We introduce a coarse substructured space $\overline{V}_0\subset \overline{V}$, a restriction operator $\overline{R}_0:\overline{V}\rightarrow \overline{V}_0$ and a prolongation operator $\overline{P}_0:\overline{V}_0\rightarrow \overline{V}$. We define the coarse substructured function as 
\begin{equation}\label{eq:two_level_sub_coarse}
\overline{\mathcal{F}}_{0}(\vv_0):=\overline{R}_0\overline{\mathcal{F}}(\overline{P}_0(\vv_0)), \quad\forall \vv_0\in \overline{V}_0.
\end{equation} 
From the definition it follows that
$J_{\overline{\mathcal{F}}_0}(\vv_0)=\overline{R}_0\JFS(\overline{P}_0\vv_0)\overline{P}_0$,
$\forall \vv_0\in \overline{V}_0$. There is a profound difference
between two-level nonlinear RAS and two-level nonlinear SRAS: in
the first one (Algorithm \ref{alg:Two-level_RAS}), the coarse function
is obtained restricting the original nonlinear system $F(\uv)=0$ onto a
coarse mesh. In the substructured version, the coarse substructured
function is defined restricting the fixed point equation of nonlinear
SRAS to $\overline{V}_0$. That is, the coarse substructured
function corresponds to a coarse version of SRASPEN.  Hence, we
remark that this algorithm is exactly the nonlinear counterpart of the
linear two-level algorithm described in
\cite{ciaramella-vanzan,ciaramella-vanzan2}. Two-level nonlinear
  SRAS is then defined in Algorithm \ref{alg:Two-level_SRAS}.
\begin{algorithm}[t]
\setlength{\columnwidth}{\linewidth}
\caption{Two-level iterative nonlinear SRAS}
\begin{algorithmic}[1]
\STATE  Solve the coarse problem $\overline{\mathcal{F}}_0\left(\yv\right)=\overline{\mathcal{F}}_{0}\left(\overline{R}_0\vv^k\right)-\overline{R}_0\overline{\mathcal{F}}\left(\vv^k\right)$ and set $C^S_0\left(\vv^k\right)=\yv-\overline{R}_0\vv^k$.
\STATE Add the coarse correction to the current iterate, $\vv^{k+\frac{1}{2}}=\vv^k+\overline{P}_0C^S_0\left(\vv^k\right)$.
\STATE Compute one-step of nonlinear SRAS, $\vv^{k+1}=\sum_{j\in \mathcal{J}}\overline{G}_j\left(\vv^{k+\frac{1}{2}}\right)$.
\STATE Repeat steps 1 to 3 until convergence.
\end{algorithmic}\label{alg:Two-level_SRAS}
\end{algorithm}
As in the linear case, numerical experiments will show that
two-level iterative nonlinear SRAS exhibits faster convergence in
terms of iteration counts compared to two-level nonlinear
  RAS. However, we remark that evaluating $F_0$ is rather cheap,
while evaluating $\overline{\mathcal{F}}_0$ could be quite expensive
as it requires to perform subdomain solves on the fine mesh. One
possible improvement is to approximate $\overline{\mathcal{F}}_0$
replacing $\overline{\mathcal{F}}$ in its definition with another
function which performs subdomain solves on a coarse mesh. Further, we
emphasize that a prerequisite of any domain decomposition method is
that the subdomain solves are cheap to compute in a high performance
parallel implementation, so that in such a setting evaluating
$\overline{\mathcal{F}}_0$ needs to be cheap as well.
\subsection{Two-level preconditioners for Newton's method}
Once we have defined the two-level iterative methods, we are ready to introduce the two-level versions of RASPEN and SRASPEN.
The fixed point equation of two-level nonlinear RAS is
\begin{equation}\label{eq:RASPEN2L}
\begin{aligned}
\FtwoL(\uv)&:=\uv-\sum_{j\in \mathcal{J}} \widetilde{P}_jG_j(\uv+P_0C_0(\uv)) \\
&=-P_0C_0(\uv)-\sum_{j\in \mathcal{J}} \widetilde{P}_jC_j(\uv+P_0C_0(\uv))=0,
\end{aligned}
\end{equation}
where we have introduced the correction operators $C_j(\uv):=G_j(\uv)-R_j\uv$.
Thus two-level RASPEN defined in \cite{dolean2016nonlinear} consists in applying Newton's method to the fixed point equation \eqref{eq:RASPEN2L}.

Similarly, the fixed point equation of two-level nonlinear SRAS is 
\begin{equation}\label{eq:SRASPEN2L}
\FtwoLS(\vv):=\vv-\sum_{j\in \mathcal{J}} \overline{G}_j\left(\vv+\overline{P}_0\overline{C}_0(\vv)\right)=-\overline{P}_0\overline{C}_0(\vv)-\sum_{j\in \mathcal{J}} \overline{C}_j\left(\vv+\overline{P}_0\overline{C}_0(\vv)\right)=0,
\end{equation}
where the correction operators $\overline{C}_j$ are defined as
$\overline{C}_j(\vv):=\overline{G}_j(\vv)-\overline{R}\widetilde{P}_jR_j\overline{P}\vv$.
Two-level SRASPEN consists in applying Newton's method to the
fixed point equation \eqref{eq:SRASPEN2L}.

\section{Numerical results}\label{sec:Num_section}

We discuss three different examples in this section to illustrate
our theoretical results. In the first example, we consider a linear
problem where we study the GMRES performance when GMRES is
applied to the preconditioned volume system and the preconditioned
substructured system. In the next two examples, we present numerical
results in order to compare Newton's method, NKRAS
  \cite{cai2011inexact}, nonlinear RAS, nonlinear SRAS, RASPEN, and
  SRASPEN for the solution of a one-dimensional Forchheimer
equation and for a two-dimensional nonlinear diffusion equation.
\subsection{Linear example}

We consider the unit cube $\Omega:=(0,1)^3$ decomposed into $N$ equally-sized bricks
with overlap, each discretized with $27000$ degrees of
freedom. The size of the overlap is $\delta:= 4\times h$.  In Table
\ref{Tab:3D} we study the computational effort and memory
required by GMRES when applied to the preconditioned volume system
\eqref{eq:precond_vol} (GMRES-RAS) and to the preconditioned
substructured system \eqref{eq:precond_sub} (GMRES-SRAS).  We let the
number of subdomains grow, while keeping their sizes constant, that is
the global problem becomes larger as $N$ increases.  We report the
computational times to reach a relative residual smaller than
$10^{-8}$, and the number of gigabytes required to store the
orthogonal matrices of the Arnoldi iteration, both for the volume and
substructured implementations.

\begin{table}[]
\centering
\begin{small}
\setlength{\tabcolsep}{5pt}
\begin{tabular}{ c | c c c }
$N_v(N)-\overline{N}$ &  729000(27)-92944 & 1728000(64)-246456 &   3375000(125)-511712  \\ \hline
GMRES-RAS & 34.31 &  134.71  &  401.27  \\
GMRES-SRAS & 33.08  & 132.81 & 393.78 \\ \vspace{0.1cm}
\end{tabular}
\end{small}\qquad
\begin{small}
\begin{tabular}{ c | c c c }
$N_v(N)-\overline{N}$ &  729000(27)-92944 & 1728000(64)-246456 &   3375000(125)-511712   \\ \hline
GMRES-RAS & 0.09 &  0.34  &  0.81  \\
GMRES-SRAS & 0.01  &  0.05 & 0.12 \\
\end{tabular}
\end{small}
\caption{On the top, time in seconds required by GMRES-RAS and
  GMRES-SRAS to reach a relative error smaller than $10^{-8}$ for
  increasingly larger problems. At the bottom, memory use
  expressed in gigabytes to store the Arnoldi orthogonal matrices in
  both GMRES implementations.}
\label{Tab:3D}
\end{table}

Table \ref{Tab:3D} shows that GMRES applied to the preconditioned
substructured system is slightly faster in terms of computational time
compared to the volume implementation. This advantage becomes
more evident as the global problem becomes larger. We emphasize that
GMRES required the same number of iterations to reach the tolerance
for both methods in all cases considered. Thus the faster time to
  solution of GMRES-SRAS is due to the smaller number of floating
point operations that GMRES-SRAS has to perform since the
orthogonalization steps are performed in a much smaller
space. Furthermore, GMRES-SRAS significantly outperforms GMRES-RAS in
terms of memory requirements; in this particular case, GMRES-SRAS
computes and stores orthogonal matrices which are about seven times
smaller than the ones used by GMRES-RAS.

\subsection{Forchheimer equation in 1D}
Forchheimer equation is an extension of the Darcy equation for high flow rates, where the linear relation between the flow velocity and the gradient flow does not hold anymore.
In a one dimensional domain $\Omega:=(0,1)$, the Forchheimer model is 
\begin{equation}\label{eq:Forch}
\begin{aligned}
&q(-\lambda(x)u(x)^\prime))^\prime = f(x)\quad \text{in } \Omega,\\
&u(0)=u_L\quad \text{and}\quad u(1)=u_R,
\end{aligned}
\end{equation}
where $u_L,u_R\in \mathbb{R}$, $\lambda(x)$ is a positive and bounded
permeability field and
$q(y):=\sign(y)\frac{-1+\sqrt{1+4\gamma|y|}}{2\gamma}$, with
$\gamma>0$. To discretize \eqref{eq:Forch}, we use the finite volume
scheme described in detail in \cite{dolean2016nonlinear}. In our
numerical experiments, we set $\lambda(x)=2+\cos(5\pi x)$,
$f(x)=50\sin(5\pi x)e^x$, $\gamma=1$, $u(0)=1$ and $u(1)=e^1$. The
solution field $u(x)$ and the force field $f(x)$ are shown in Figure
\ref{Fig:Forch}.
\begin{figure}
\centering
\includegraphics[width=0.49\textwidth]{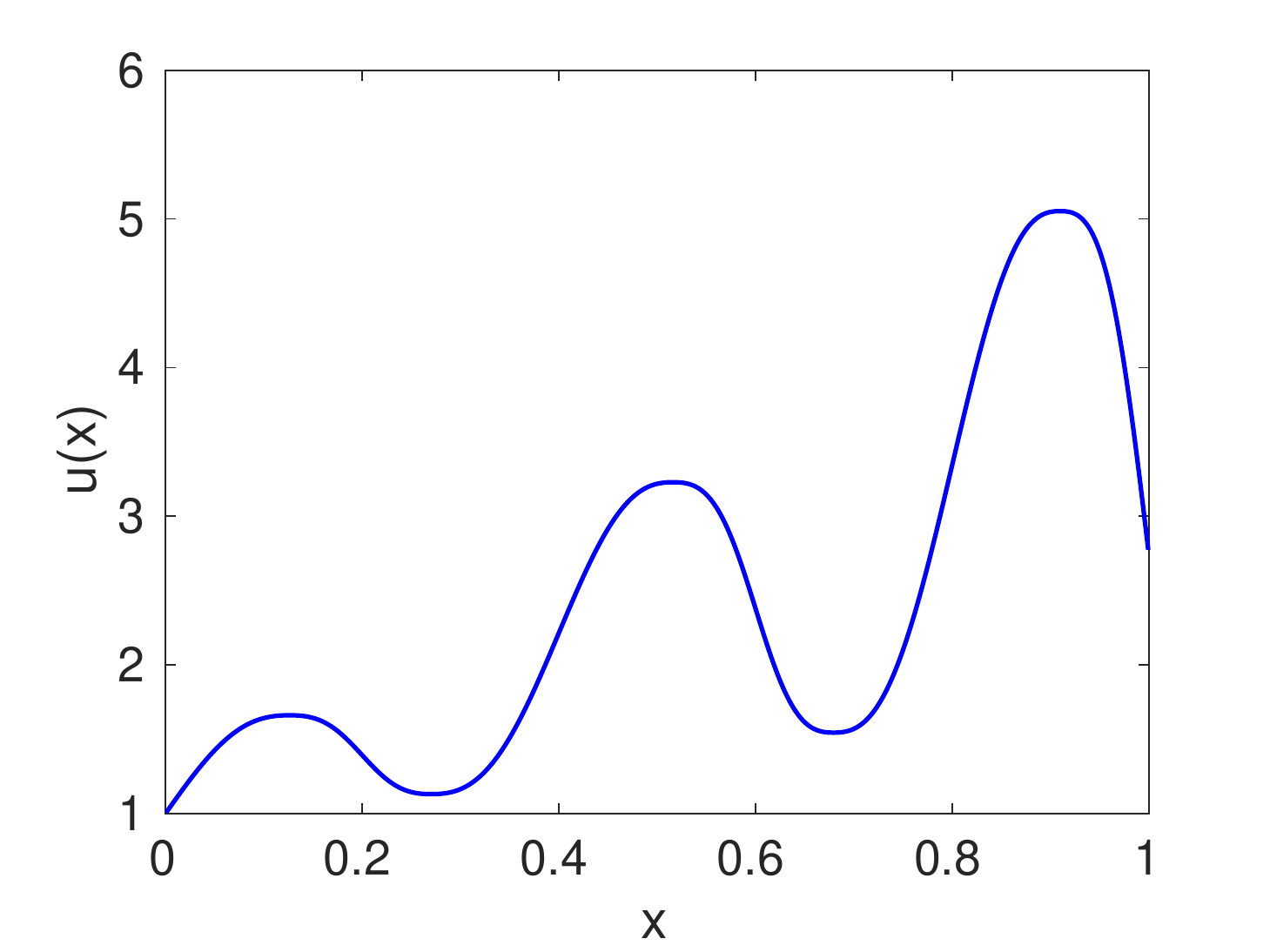}
\includegraphics[width=0.49\textwidth]{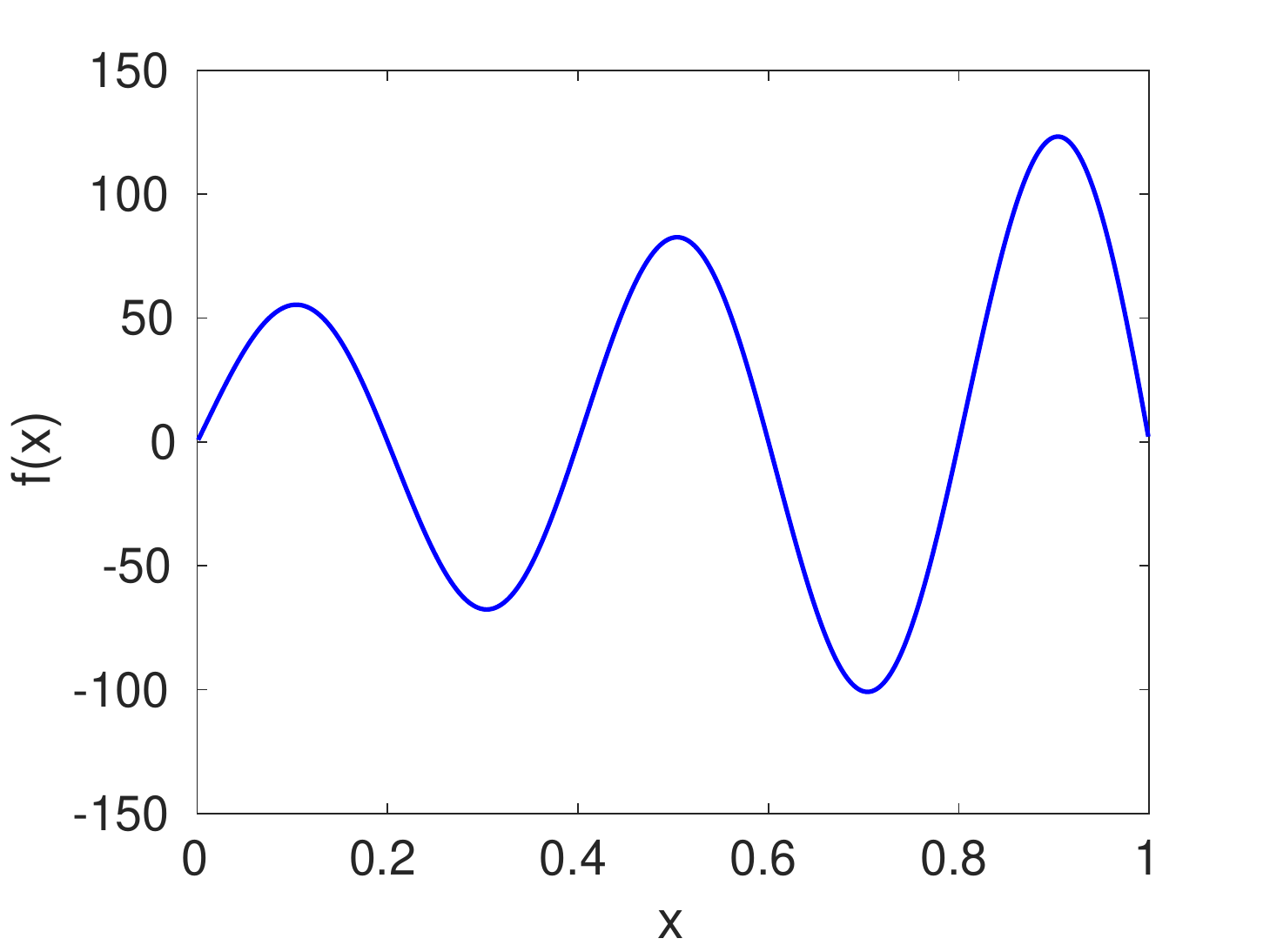}
\caption{Solution field $u(x)$ of Forchheimer equation (left panel) and force term $f(x)$ (right panel).}\label{Fig:Forch}
\end{figure}
We then study the convergence behavior of our different
methods. Figure \ref{Fig:Forch_Conv} shows how the relative error
decays for the different methods and for a decomposition into 20
subdomains (left panel) and 50 subdomains (right panel). The initial
guess is equal to zero for all these methods.
\begin{figure}
\centering
\includegraphics[width=0.49\textwidth]{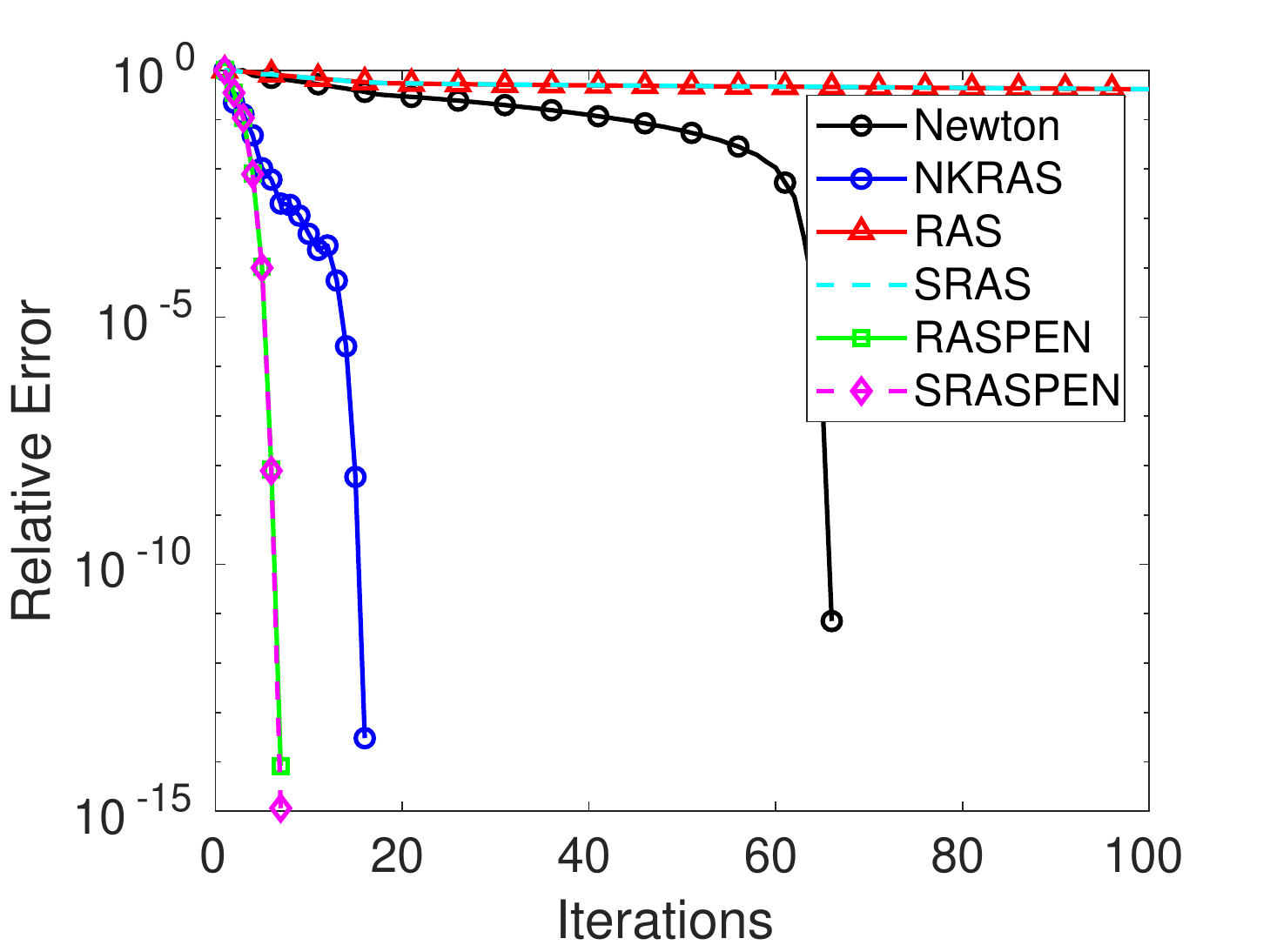}
\includegraphics[width=0.49\textwidth]{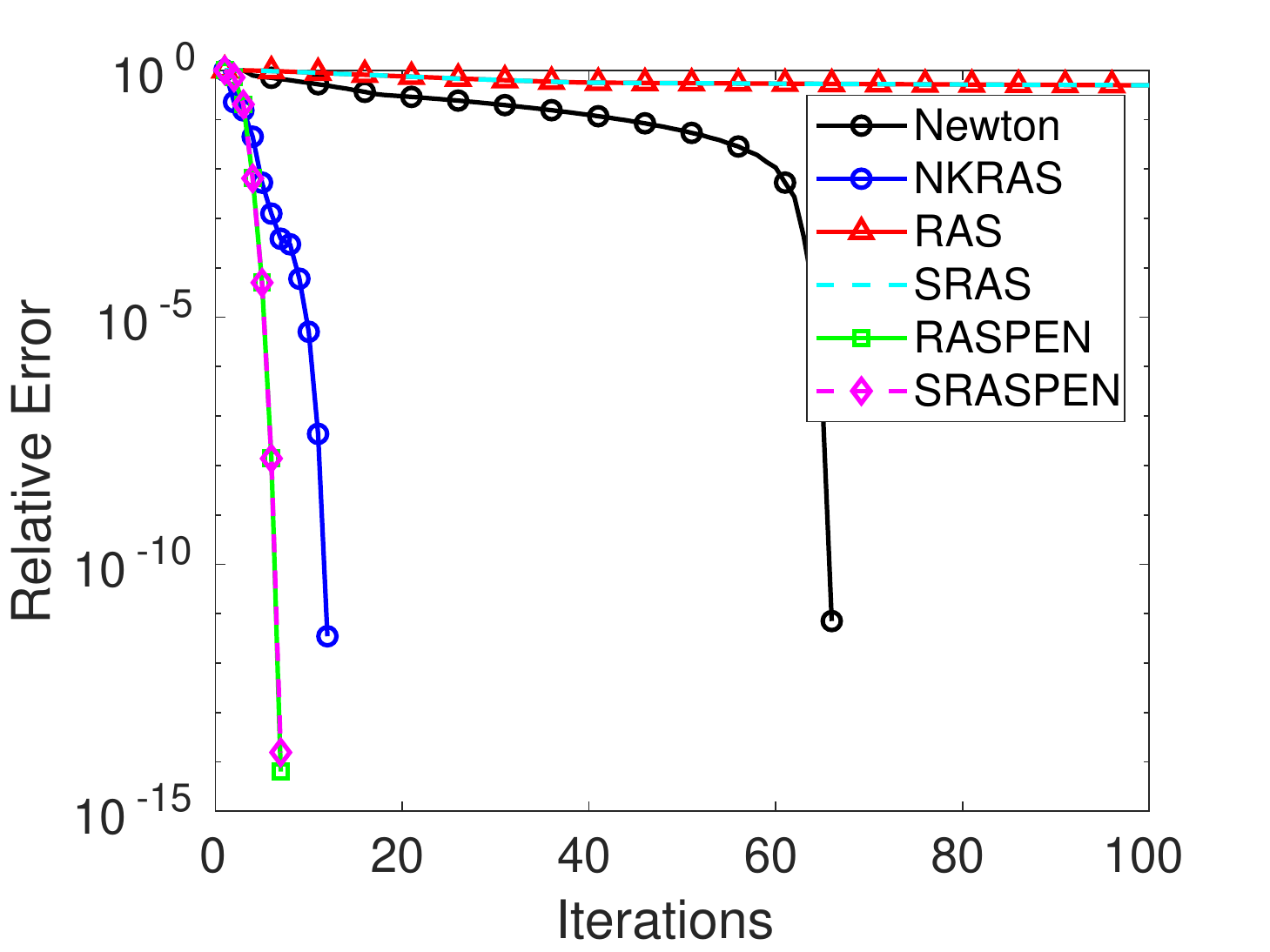}
\caption{Convergence behavior for Newton's method, NKRAS,
  nonlinear RAS, nonlinear SRAS, RASPEN and SRASPEN
  applied to Forchheimer equation. On the left, the
  simulation refers to a decomposition into 20 subdomains while on the
  right we consider 50 subdomains. The mesh size is $h=10^{-3}$ and
  the overlap is $8h$.}\label{Fig:Forch_Conv}
\end{figure}
Both plots in Figure \ref{Fig:Forch_Conv} show that the
convergence rate of iterative nonlinear RAS and nonlinear SRAS
is the same and very slow. As expected, NKRAS
with line search converges better than Newton's method. Further,
RASPEN and SRASPEN converge in the same number
of outer Newton iterations. Moreover, it seems that the convergence
of RASPEN and SRASPEN is not affected by the number of
subdomains. However, these plots do not tell the whole story, as one
should focus not only on the number of iterations but also on the cost
of each iteration. To compare the cost of an iteration of RASPEN
and SRASPEN, we have to distinguish two cases, that is, if one
solves the Jacobian system directly or with some Krylov methods, e.g.,
GMRES. First, suppose that we want to solve the Jacobian system with a
direct method and thus we need to assemble and store the
Jacobians. From the expressions in equation
\eqref{eq:Jacobian_SRASPEN} we remark that the assembly of the
Jacobian of RASPEN requires $N\times N_v$ subdomain solves,
where $N$ is the number of subdomains and $N_v$ is the number of
unknowns in volume. On the other hand, the assembly of the Jacobian of
SRASPEN requires $N\times \overline{N}$ solves, where
$\overline{N}$ is the number of unknowns on the substructures and
$\overline{N}\ll N_v$. Thus, while the assembly of $\JF$ is
prohibitive, it can still be affordable to assemble $\JFS$. Further,
the direct solution of the Jacobian system is feasible as $\JFS$ has
size $\overline{N}\times \overline{N}$. Suppose now that we solve the
Jacobian systems with GMRES. Let us indicate with $I(k)$ and $I^S(k)$
the number of GMRES iterations to solve the volume and substructured
Jacobian systems at the $k$-th outer Newton iteration. Each GMRES
iteration requires $N$ subdomain solves which can be performed in
parallel. In our numerical experiment, we have observed that generally
$I^S(k)\leq I(k)$, with $I(k)-I^S(k)\approx 0,1,2$, that is GMRES
requires the same number of iterations or slightly less to solve the
substructured Jacobian system compared to the volume one.

To better compare these two methods, we follow
\cite{dolean2016nonlinear} and introduce the quantity $L(n)$ which
counts the number of subdomain solves performed by these two
methods till iteration $n$, taking into account the advantages of a
parallel implementation. We set $L(n)=\sum_{k=1}^n L_{in}^k + I(k)$,
where $L_{in}^k$ is the maximum over the subdomains of the number of
Newton iterations required to solve the local subdomain problems
at iteration $k$. The number of linear solves performed by GMRES
should be $I(k)\times N$, but as the $N$ linear solves can be
performed in parallel, the total cost of GMRES corresponds
approximately to $I(k)$ linear solves.  Figure \ref{Fig:Forch_LN}
shows the error decay as a function of $L(n)$. We note that the two
methods require approximately the same computational cost and SRASPEN
is slightly faster.
\begin{figure}
\centering
\includegraphics[width=0.49\textwidth]{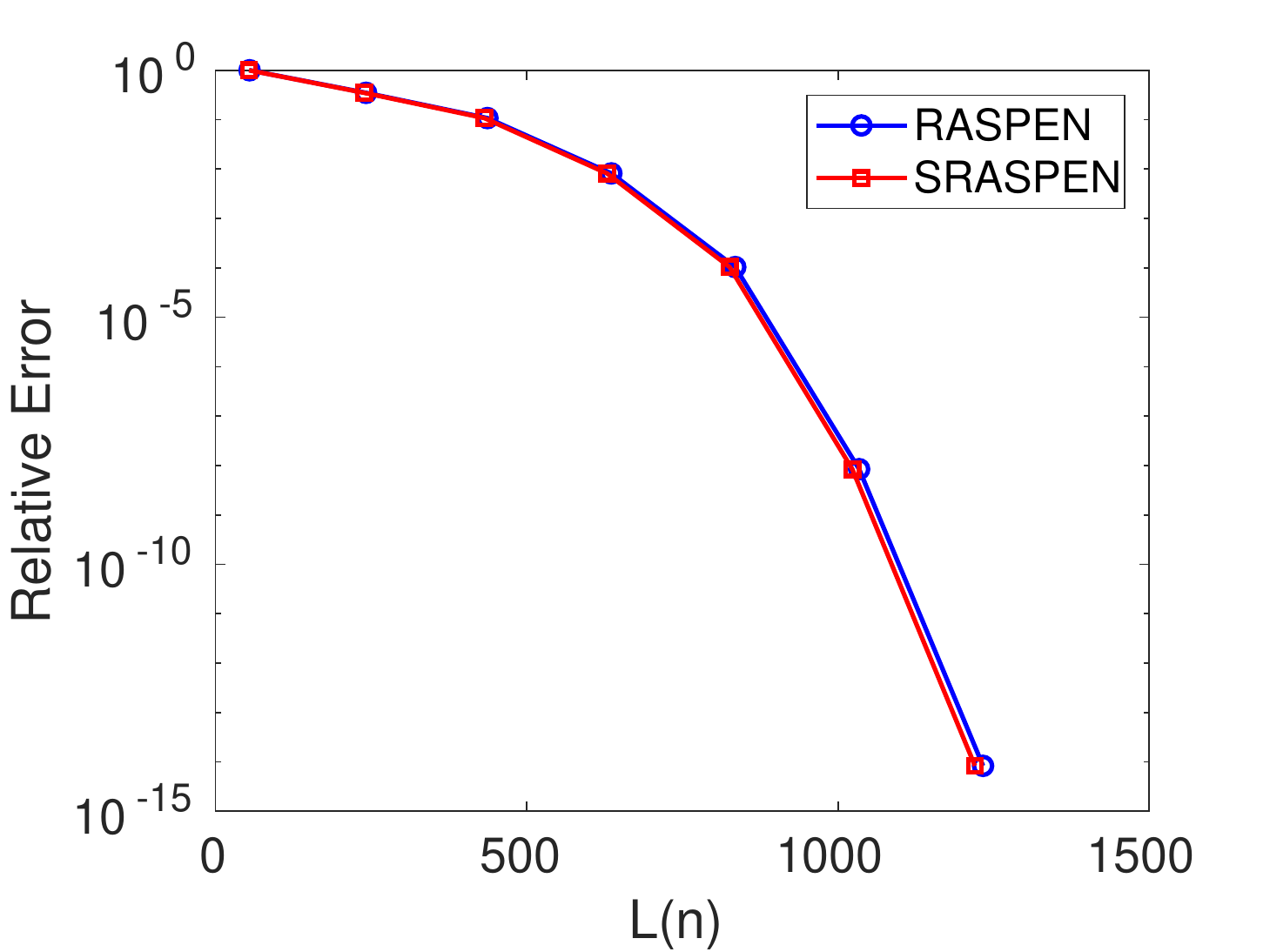}
\includegraphics[width=0.49\textwidth]{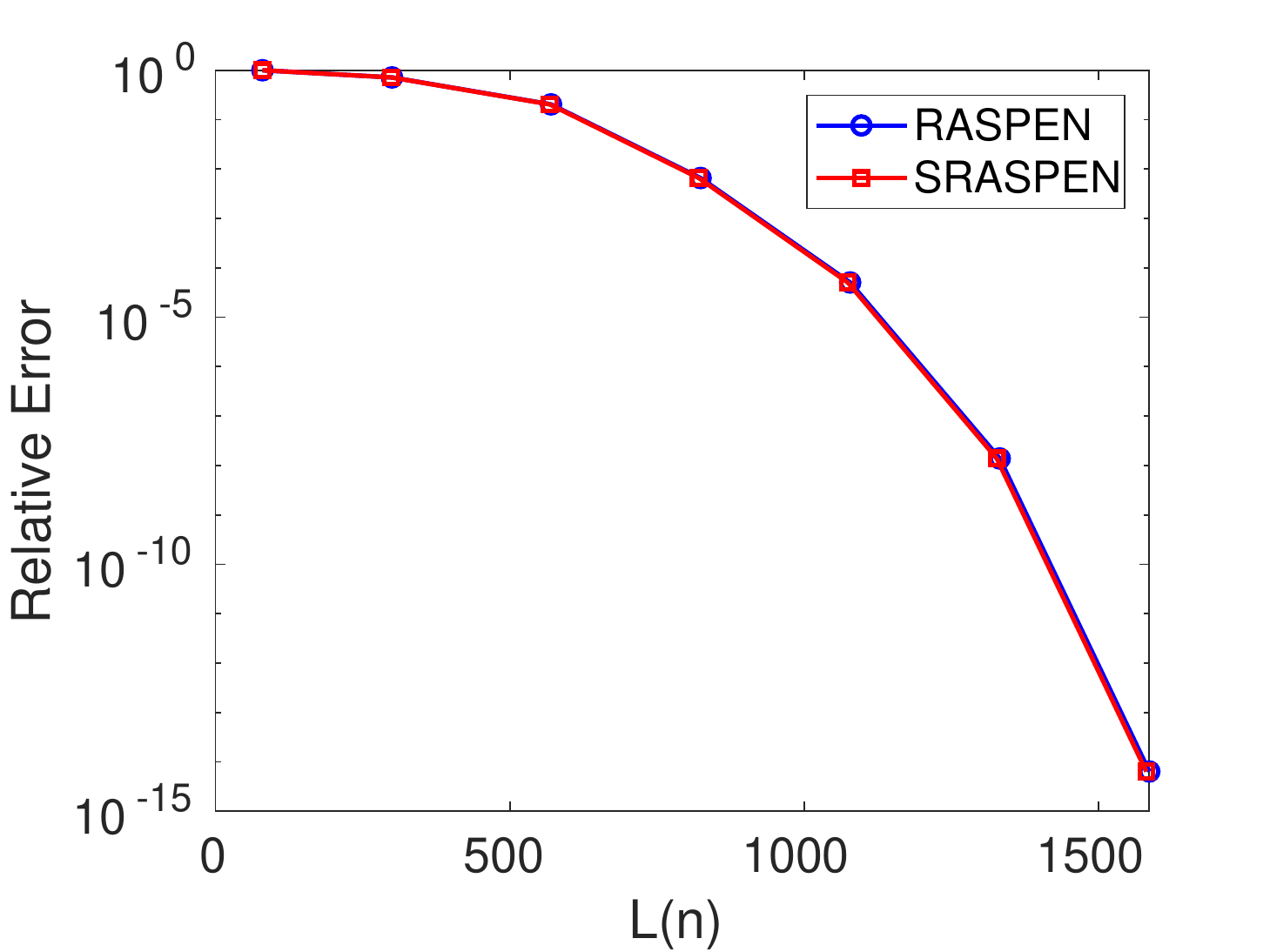}
\caption{Relative error decay for RASPEN and SRASPEN applied to
  Forchheimer equation with respect to the number of linear
  solves. On the left, the simulation refers to a decomposition into
  20 subdomains while on the right we consider 50 subdomains. The mesh
  size is $h=10^{-3}$.}\label{Fig:Forch_LN}
\end{figure}
For the decomposition into 50 subdomains, RASPEN requires on average
91.5 GMRES iterations per Newton iteration, while SRASPEN
requires an average of 90.87 iterations. The size of the substructured
space $\overline{V}$ is $\overline{N}=98$. For the decomposition into
20 subdomains, RASPEN requires an average of 40 GMRES
iterations per Newton's iteration, while SRASPEN needs 38
iterations. The size of $\overline{V}$ is $\overline{N}=38$, which
means that GMRES reaches the given tolerance of $10^{-12}$ after
exactly $\overline{N}$ steps, which is the size of the substructured
Jacobian. Under these circumstances, it can be convenient to actually
assemble $\JFS$, as it requires $\overline{N}\times N$ subdomain
solves which is the total cost of GMRES. Furthermore, the
$\overline{N}\times N$ subdomain solves are embarrassingly parallel,
while the $\overline{N}\times N$ solves of GMRES can be parallelized
in the spatial direction, but not in the iterative one.  As future
work, we believe it will be interesting to study the convergence of a
Quasi-Newton method based on SRASPEN, where one assembles
the Jacobian substructured matrix after every few outer Newton
iterations, reducing the overall computational cost.

As a final remark, we specify that Figure \ref{Fig:Forch_LN} has been
obtained setting a zero initial guess for the nonlinear subdomain
problems. However, at the iteration $k$ of RASPEN one can use the
subdomain restriction of the updated volume solution, that is $R_j
\uv^{k-1}$, which has been obtained by solving the volume Jacobian
system at iteration $k-1$, and is thus generally a better initial
guess for the next iteration. On the other hand in SRASPEN, one
could use the subdomain solutions computed at iteration $k-1$,
i.e. $\uv_i^{k-1}$, as initial guess for the nonlinear subdomain
problems, as the substructured Jacobian system corrects only the
substructured values.  Numerical experiments showed that with this
particular choice of initial guess for the nonlinear subdomain
problems, SRASPEN requires generally more Newton iterations to
solve the local problems. In this setting, there is not a method that
is constantly faster than the other as it depends on a delicate
trade-off between the better GMRES performance and the need to perform
more Newton iterations for the nonlinear local problems in
SRASPEN.

\subsection{Nonlinear Diffusion}
In this subsection we consider the nonlinear diffusion problem on a square domain $\Omega:=(0,1)^2,$
\begin{equation}
\begin{aligned}
-\nabla\cdot \left(1+u(x)^2\right)\nabla u(x)&=f,\quad\text{ in } \Omega,\\
u(x)&=g(x)\quad \text{on }\partial \Omega,
\end{aligned}
\end{equation}
where the right hand side $f$ is chosen such that $u(x)=\sin(\pi x)\sin(\pi y)$ is the exact solution. We start all these methods with an initial guess $u^0(x)=10^5$, so that we start far away from the exact solution, and hence Newton's method exhibits a long plateau before quadratic convergence begins. 
\begin{figure}[]
\centering
\includegraphics[width=0.49\textwidth]{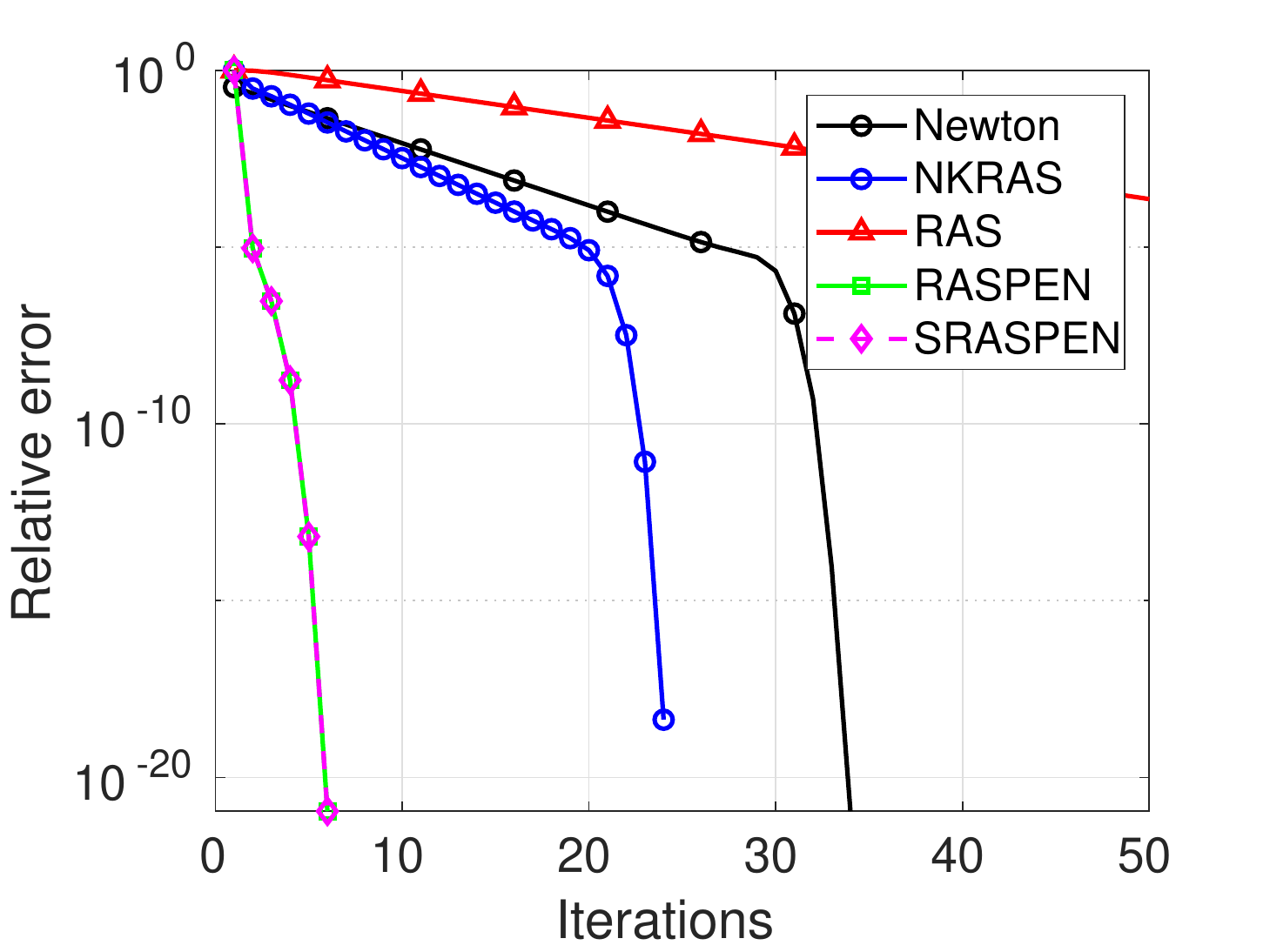}
\includegraphics[width=0.49\textwidth]{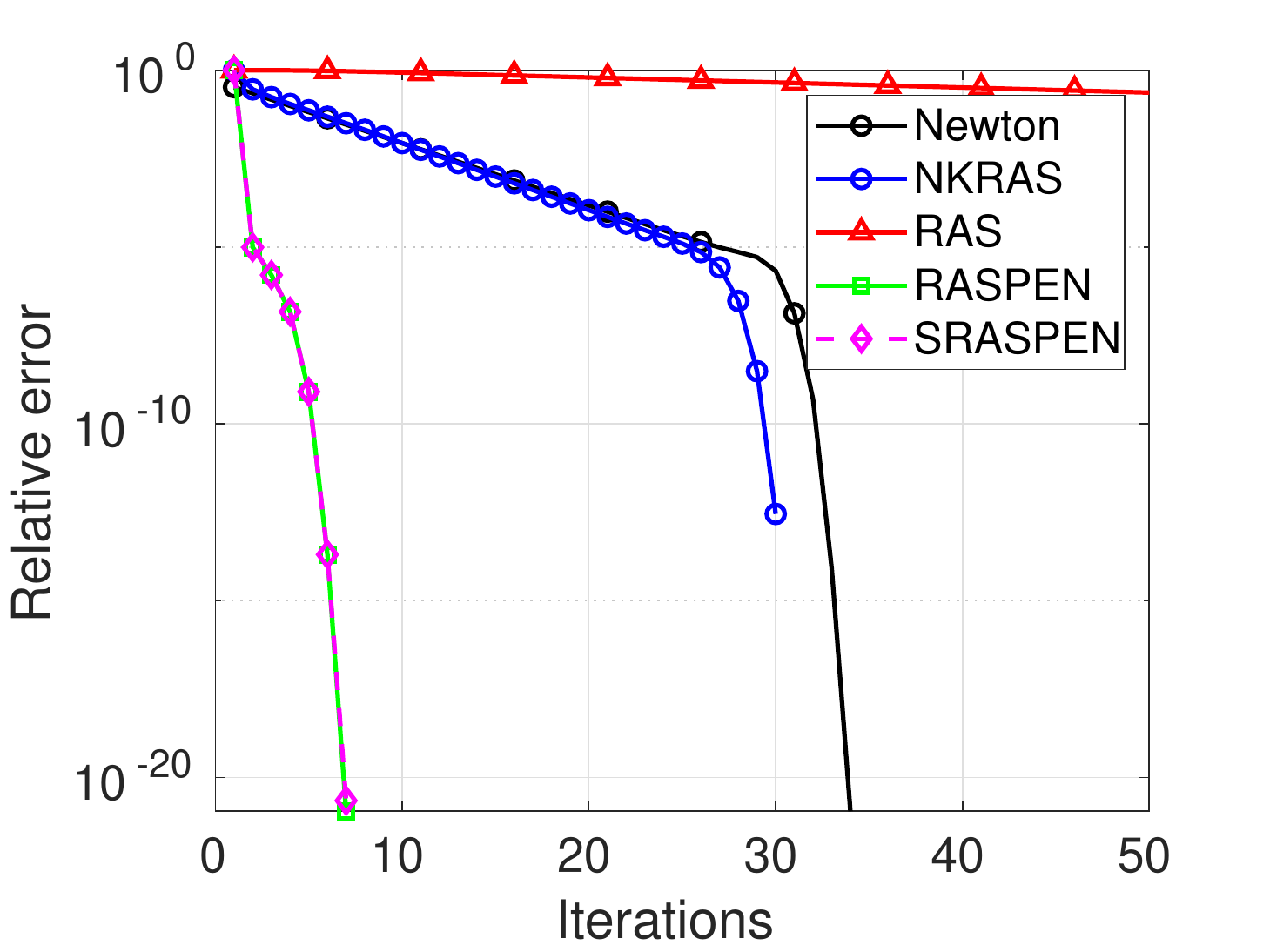}\\
\includegraphics[width=0.49\textwidth]{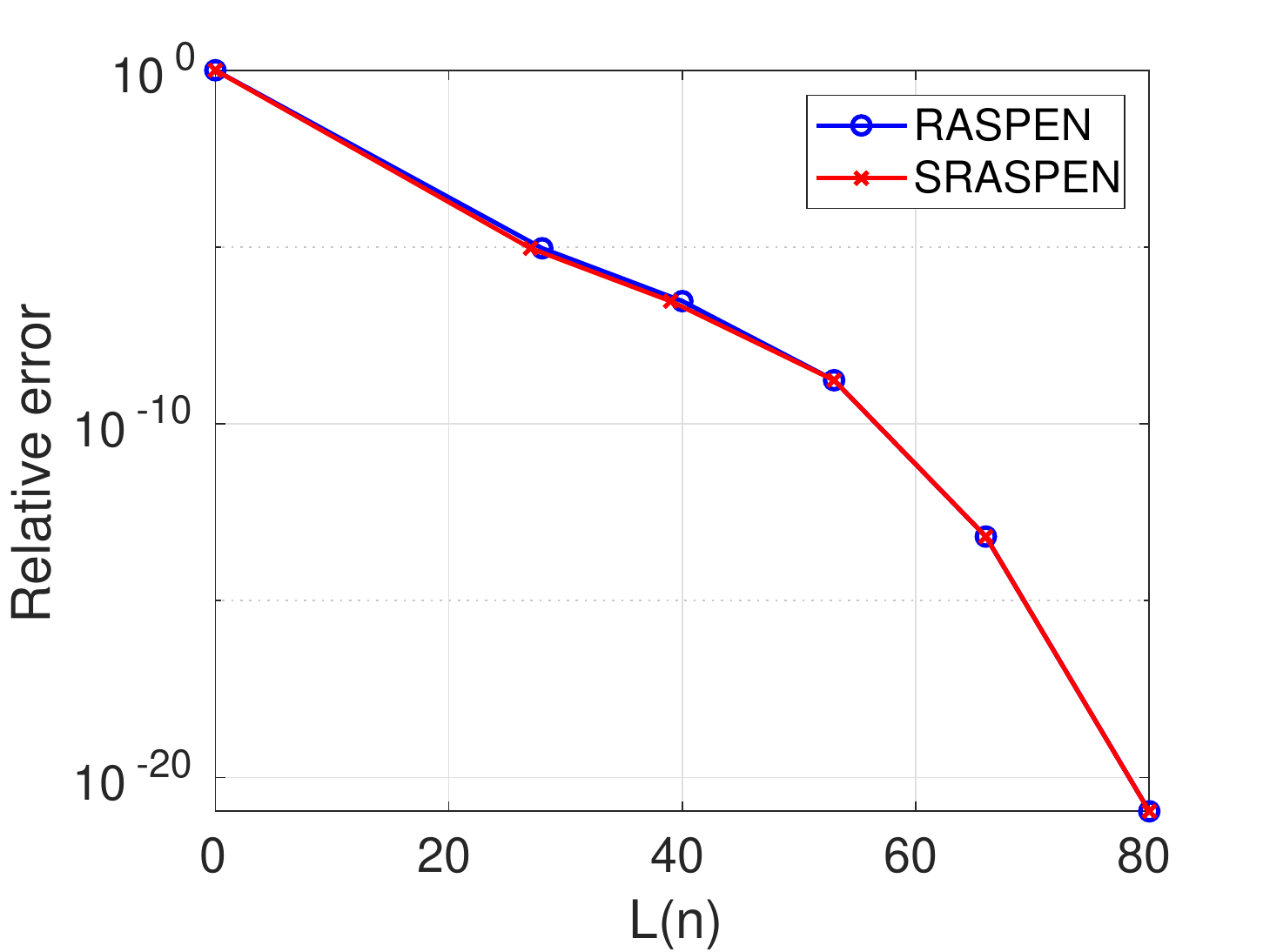}
\includegraphics[width=0.49\textwidth]{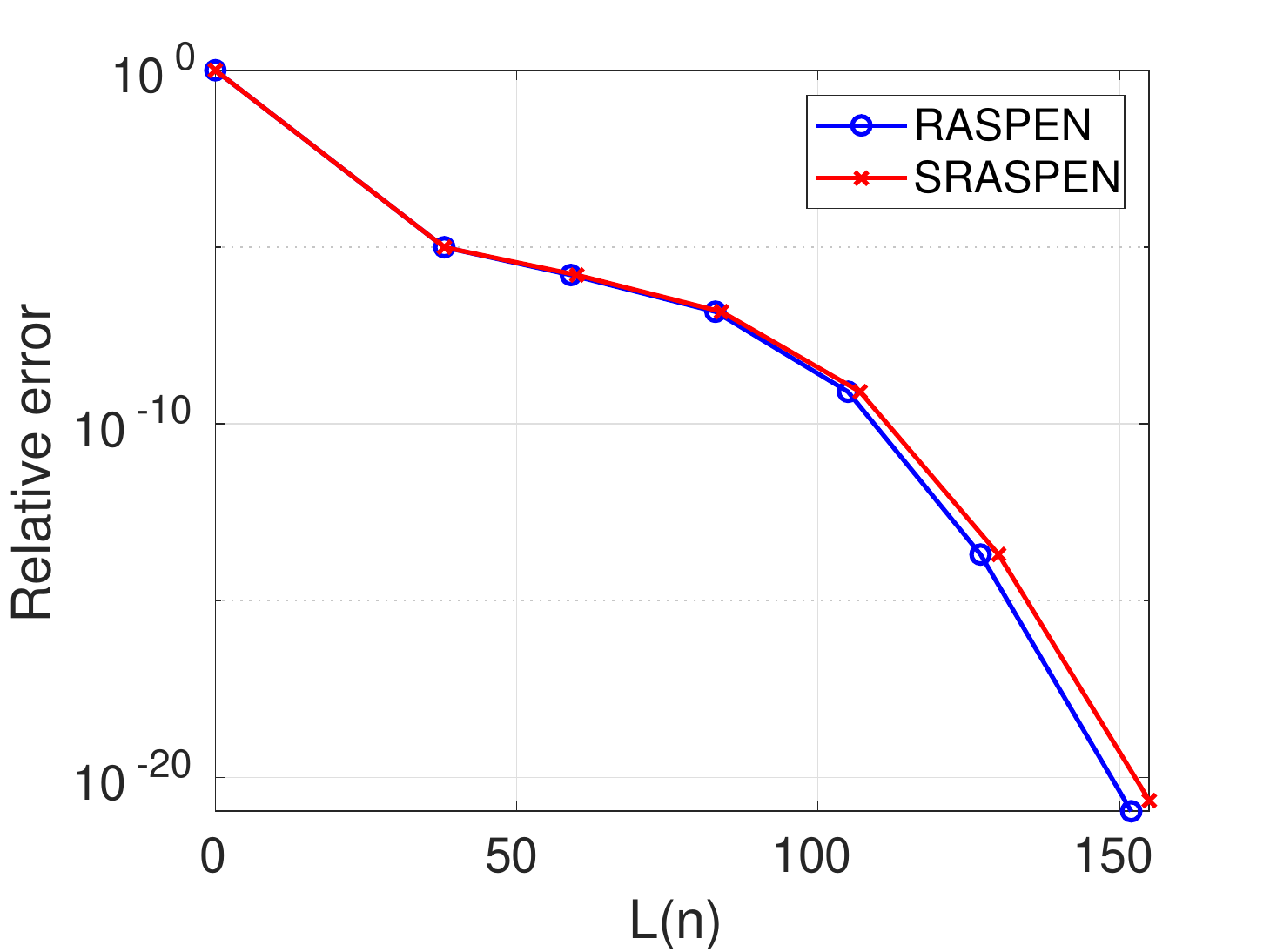}
\caption{Relative error decay versus the number of iterations (top
  row) and error decay versus the number of linear solves (bottom
  row). The left figures refer to a decomposition into four subdomains,
  the right figures to a decomposition into 25 subdomains. The
  mesh size is $h=0.012$ and the overlap is $8h$.}\label{Fig:NLD_it}
\end{figure}
Figure \ref{Fig:NLD_it} shows the convergence behavior for the
different methods as function of the number of iterations and the
number of linear solves. The average number of GMRES iterations is
8.1667 for both RASPEN and SRASPEN for the four subdomain
decomposition. For a decomposition into 25 subdomains, the average
number of GMRES iterations is 19.14 for RASPEN and 19.57 for
SRASPEN. We remark that as the number of subdomains increases,
GMRES needs more iterations to solve the Jacobian system. This is
consistent with the interpretation of \eqref{eq:Jacobian_SRASPEN} as a
Jacobian matrix $J\left(\uv^{(j)}\right)$ preconditioned by the additive
operator $\sum_{j\in \mathcal{J}} \left(R_j
J\left(\uv^{(j)}\right)P_j\right)^{-1}$; We expect this preconditioner
not to be scalable since it does not involve a coarse correction.

We conclude this section by showing the convergence behavior for the
two-level variants of nonlinear RAS, nonlinear SRAS, RASPEN, and
SRASPEN. We use a coarse grid in volume taking half of the points in
$x$ and $y$, and a coarse substructured grid taking half of the
unknowns as depicted in Figure \ref{fig:subdomain}. The interpolation
and restriction operators $P_0,R_0,\overline{P}_0$ and
$\overline{R}_0$ are the classical linear interpolation and fully
weighting restriction operators defined in Section
\ref{Sec:two-level}. From Figure \ref{Fig:NLD_twolevel},
\begin{figure}[]
\centering
\includegraphics[width=0.49\textwidth]{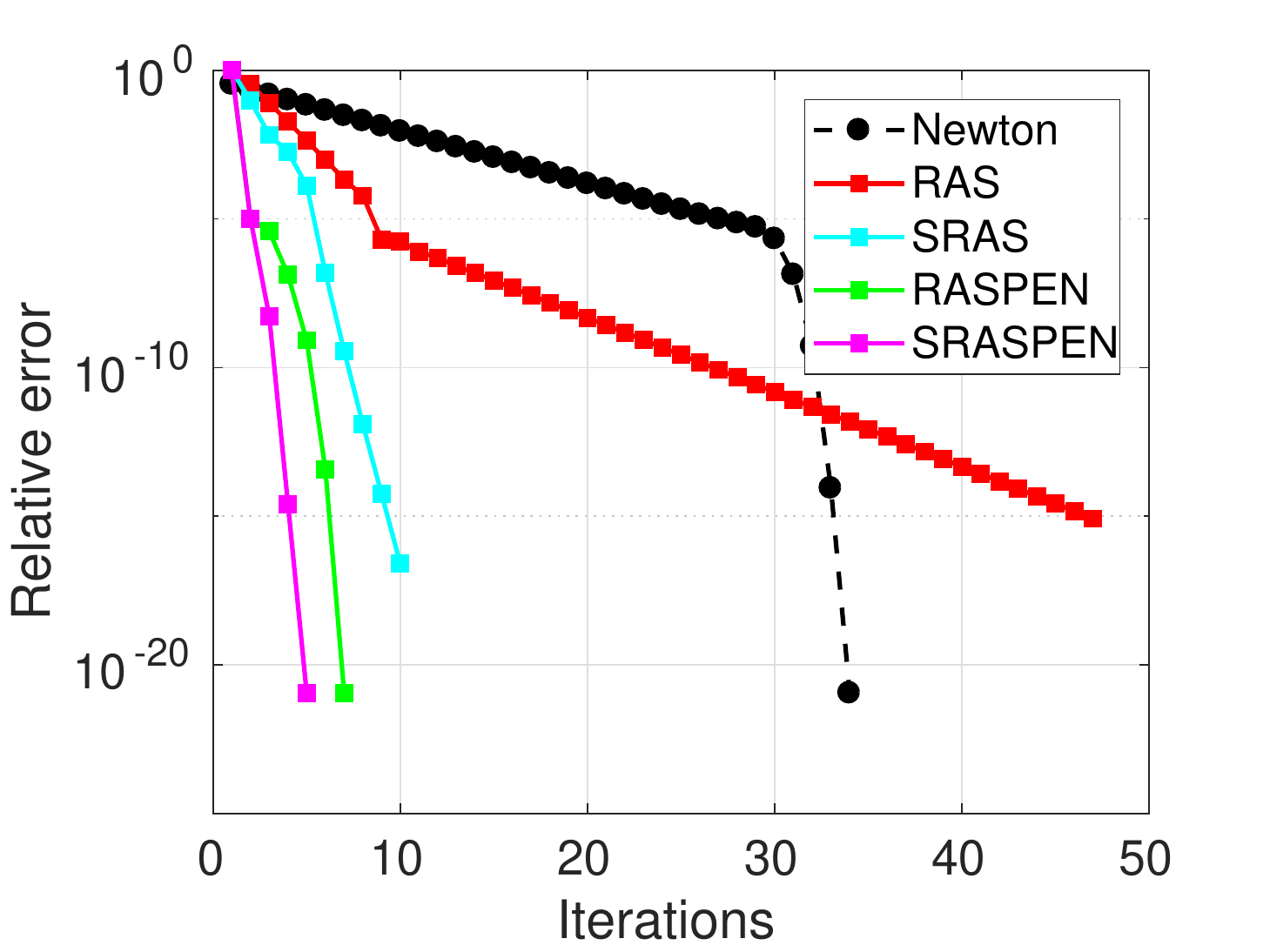}
\includegraphics[width=0.49\textwidth]{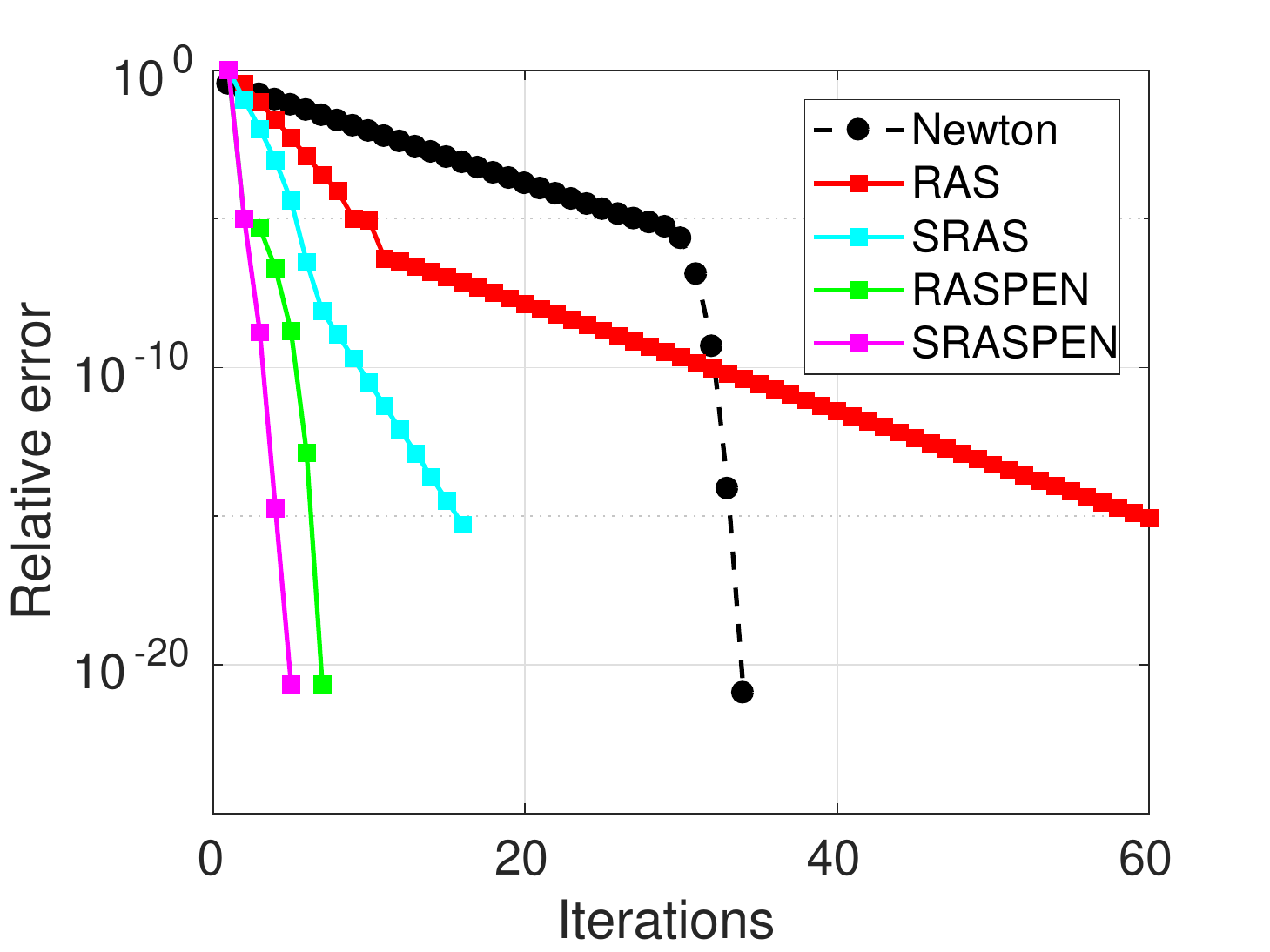}\\
\caption{Relative error decay versus the number of iterations for
  Newton's method, iterative two-level nonlinear RAS and SRAS, and the
  two-level variants of RASPEN and SRASPEN. The left figure
  refers to a decomposition into 4 subdomains, while the right figure
  refers to a decomposition into 16 subdomains. The mesh size is
  $h=0.012$ and the overlap is $4h$.}\label{Fig:NLD_twolevel}
\end{figure}
we note that two-level nonlinear SRAS is much faster than
two-level nonlinear RAS, and this observation is in agreement
with the linear case treated in
\cite{ciaramella-vanzan,ciaramella-vanzan-spectral}. Since the
two-level iterative methods are not equivalent, we also remark that
two-level SRASPEN shows a better performance than two-level
  RASPEN in terms of iteration count. As the one-level smoother
is the same in all methods, the better convergence of the
substructured methods implies that the coarse equation involving
$\overline{\mathcal{F}}_0$ provides a much better coarse correction
than the classical volume one involving $F_0$.

Even though the two-level substructured methods are faster in terms of
iteration count, the solution of the FAS problem involving
$\overline{\mathcal{F}}_0=\overline{R}_0\overline{\mathcal{F}}(\overline{P}_0(\vv_0))$
is rather expensive as it requires to evaluate twice the substructured
function $\overline{\mathcal{F}}$ (each evaluation requires subdomain
solves) to compute the right hand side, to solve a Jacobian system
involving $J_{\overline{\mathcal{F}}_0}$, and to evaluate
$\overline{\mathcal{F}}$ on the iterates, which again require the
solution of subdomain problems. Unless one has a fully parallel
implementation available, the coarse correction involving
$\overline{\mathcal{F}}_0$ is doomed to represent a bottleneck.

\section{Conclusions}      

We presented for the first time an analysis of the effects of
substructuring on RAS when it is applied as an iterative solver
and as a preconditioner. We proved that iterative RAS and iterative
SRAS converge at the same rate, both in the linear and nonlinear
case. For the nonlinear case, we showed that the preconditioned
methods, namely RASPEN and SRASPEN also have the same rate of
convergence as they produce the same iterates once these are
restricted to the interfaces.  Surprisingly, the equivalence
between volume and substructured RAS breaks down when they are
considered as preconditioners for Krylov methods. We showed that the
Krylov spaces are equivalent, once the volume one is restricted
to the substructure, however we obtained that the iterates
are different by carefully deriving the least squares problems
solved by GMRES.  Our analysis shows that GMRES should always be
applied to the substructured system as it converges similarly when
applied to the volume formulation, but needs much less
memory. This allows us to state that, while nonlinear RASPEN
  and SRASPEN produce the same iterates, SRASPEN has advantages when
  solving the Jacobian system, either because the use of a direct
  solve is feasible, or because the Krylov method can work at the
  substructured level.  Finally, we introduced substructured
  two-level nonlinear SRAS and SRASPEN, and showed numerically that
  these methods have better convergence properties than their volume
  counterparts in terms of iteration count, although they are
  quite expensive in the present form per iteration.  Future
  efforts will be in the direction of approximating
  $\overline{\mathcal{F}}_0$, by replacing the function
  $\overline{\mathcal{F}}$, which is defined on a fine mesh, with an
  approximation on a very coarse mesh, thus reducing the overall cost
  of the substructured coarse correction, or by using spectral coarse
  spaces.

\bibliographystyle{siamplain}
\bibliography{references}

\end{document}